\documentclass{article}

\usepackage{amsmath,amssymb,amsfonts,amstext}
\usepackage{amsthm}
\usepackage{url,xspace}

\usepackage[usenames,dvipsnames]{color}
\usepackage{enumerate}

\usepackage{verbatim}
\usepackage{subcaption}
\usepackage{graphicx}
 \usepackage{hyperref}
\usepackage{cleveref}
\usepackage{thmtools}
\usepackage{array,multirow}
\usepackage[utf8]{inputenc}
\usepackage[english]{babel}
\usepackage{float}
\usepackage[numbers,sort&compress]{natbib}
\usepackage[framemethod=tikz]{mdframed}
\usepackage[toc]{appendix}
\usepackage{bbm}
\usepackage{yhmath}
\usepackage{floatpag}
\usepackage[export]{adjustbox}
\newmdenv[
  skipabove=5,
  skipbelow=5,
  innerleftmargin = .5em,
  innertopmargin = .5em,
  innerbottommargin = .5em,
  innerrightmargin = 0pt,
  rightmargin = 0pta,
  leftmargin = 2em,
  linewidth = 2pt,
  topline = false,
  rightline = false,
  bottomline = false
  ]{leftbar}
  \usepackage{natbib}
\usepackage{tikz}
\usepackage[mathscr]{eucal}
\usetikzlibrary{decorations.pathreplacing,angles,quotes}

\newcommand{\eek}{\upsilon}

\newtheorem{theorem}{Theorem}[section]
\newtheorem{lemma}{Lemma}
\newtheorem{corollary}{Corollary}
\newtheorem{proposition}{Proposition}

\newtheorem{remark}{Remark}

\newtheorem{claim}{Claim}
\newtheorem{definition}{Definition}

\newcommand{\te}{\tilde\epsilon}
\newcommand{\e}{\epsilon}
\newcommand{\R}{\mathbb{R}}
\newcommand{\N}{\mathbb{N}}
\newcommand{\fpe}{f\star\phi_\e}
\newcommand{\aep}{a_\e}
\newcommand{\be}{b_\e}
\newcommand{\gep}{{\gamma_\epsilon}}
\newcommand{\ti}{{\theta_1}}
\newcommand{\tii}{{\theta_2}}

\newcommand{\Z}{{\mathbb Z}}
\crefname{theorem}{Theorem}{Theorems}
\crefname{lemma}{Lemma}{Lemmas}
\crefname{claim}{Claim}{Claims}
\crefname{remark}{Remark}{Remarks}
\crefname{proposition}{Proposition}{Propositions}
\crefname{corollary}{Corollary}{Corollaries}
\crefname{section}{Section}{Sections}
\crefname{figure}{Figure}{Figures}
\crefname{table}{Table}{Tables}
\crefname{equation}{}{}
\crefname{definition}{Definition}{Definitions}
\crefname{construction}{Construction}{Constructions}

\title{Unbalanced Minimal Matchings for Measurable Scale Invariant Costs on
  $\mathbb{R}^d$ }

\author{Omer Angel \and Nitya Gadhiwala}

\usepackage{xcolor}

\begin{document}

\maketitle

\begin{abstract}
Given red and blue points drawn from Poisson point processes of equal
intensities, we match them by locally minimizing a cost function. In
\cref{thm:unbalcost}, we
identify all possible measurable cost functions on $\R^d$ such that the
resulting matching is scale invariant. %Then, considering
                                %thesematchings on $\R$, we include a
                                %new kind of function not considered
                                %in\cite{minimal} and prove some of
                                %their results for our new type
                                %ofunbalanced cost functions.
Our cost functions can be asymmetrical, which we refer to as
unbalanced cost functions. The investigation is motivated by
\cite{minimal} where the authors consider scale invariant cost
functions with added regularity assumptions. In \cref{compwbal}, we
discuss how different the unbalanced matchings are compared to these
more regular cost functions for points on $\R$.
\end{abstract}

\tableofcontents

\section{Introduction}

The problem of matching two discrete sets of points to each other has
been studied extensively in combinatorics, probability and
through the lens of optimal transport. In general, we are given two
locally finite collections of points $R$ and $B \in \R^d$, denoting red and
blue points and we aim to find a bijection between these
points, which also  minimizes a given cost function.

In \cite{ajtai1984optimal}, the authors sample $n$ red and $n$ blue points
uniformly from a two dimensional cube of volume $n$ aiming to minimize
the average length of an edge in the matching. They show that with high
probability there exists a matching with typical edge lengths of order
$\sqrt{\log(n)}$, and that this is the best possible rate. This result
has been extensively refined and generalized since, for different cost
functions, dimensions, distributions of red and blue points and better
asymptotics. See \cite{ledoux23,bb2013} for a survey of such results.

In the landmark paper of Gale and Shapley \cite{galeshapley}, a slightly different ``cost
function'' is studied where the red and blue points, submit lists of
preferences denoting which points of the opposite color they would
like to be matched to. A matching is said to be stable if there does
not exist a pair of red and blue points that are not matched to each other, who
prefer each other over their current partners. In \cite{galeshapley},
the authors provide
an algorithm to construct such a stable matching. When matching red and blue points on
$\R^d$, one can consider stable matchings where points prefer to be
matched to points of the opposite color which are closest to them.

Matchings where the red and blue points are sampled from Poisson point
processes of equal intensity on $\R^d$ have been studied in
\cite{holroyd2003trees,hpps,timar23,hmo24,minimal,h}.
In \cite{hpps},
the authors consider translation invariant matchings with no unmatched
points and show that for $d=1,2$, the length of an edge has infinite
first moment. And that for $d\geq 3$, there exist translation invariant matching schemes
where the matching distance has finite exponential moments. They also
compare these matchings to the Gale-Shalpey stable matching
\cite{galeshapley}.
  In \cite{h}, the author considers a cost function which is
  locally given by sums of lengths of edges and shows that such a minimal
  matching exists for $d=1$ and $d\geq 3$ but not on the strip
  $\R\times[0,1]$. In \cite{hmo24}, the authors show that for the
  above problem, there does not exist such a minimal matching in
  $d=2$. Furthermore, they show that a minimal matching doesn't exist
  in $d=2$ for a cost function given by sums of $p^{\mbox{th}}$ powers
  of the lengths of edges for any $p>1$.

In \cite{minimal}, the authors consider cost functions resulting in scale invariant matchings
between Poisson point processes on $\R$ which satisfy certain
regularity assumptions. A scale invariant matching is
one such that the matching which minimizes the cost function does so
on all scales i.e., even on zooming in and zooming out. They classify
functions which could result in scale invariant matchings on $\R$ and
describe some properties of these minimal matchings. In this paper, we only assume measurability
of the cost functions and classify what any cost function on $\R^d$
 look like in \cref{thm:unbalcost}. Our cost functions can be
asymmetric which we refer to as unbalanced cost functions. This means
that the cost of edges not only depends on the length of the edge but
also on a coefficient which may change with the orientation of the
edge. We compare the minimal matchings obtained for these unbalanced
cost functions to the balanced cost functions from \cite{minimal} in
\cref{compwbal}.

\subsection{Model and Main results}\label{sec:modelmainresults}

In this section, we introduce the model and  state the
main results.

Let $R=\{x_i\}_{i\in I_1}$ and $B=\{y_i\}_{i\in I_2}$ be two countable subsets of $\mathbb{R}^d$ denoting
the red points and the blue points respectively. We aim to classify
scale invariant
matchings between these points.  A matching $M$ between $R$ and $B$ is a set of edges, given by ordered pairs
  $\langle x,y \rangle$, such that $x\in R$, $y\in B$ and each red or
  blue point belongs to at most one edge in the matching. A matching with no unmatched
  points of any color is called a {\em perfect matching}.

We use $M(x)$ to denote the point that $x$ is matched
to, so for an edge $\langle  x,y \rangle$ with $x\in R$ and $y\in B$ we
have $M(x)=y$ and $M(y)=x$. If a point $z$ is unmatched we say that
$M(z)=\infty$, consequently, $M^{-1}(\infty)$ denotes the set
of unmatched points.  Our goal is to match these points in a way that leaves no
points unmatched and also minimizes a given cost function.

Given a function $f:\R^d\rightarrow \R$ we define the cost
of a matching as follows. For an edge $\langle x,y \rangle$ with
$x\in R$ and $y\in B$, we let $f(y-x)$ denote the cost of that edge
\footnote{ Note that the cost of an edge
depends on the relative position between the two points but not where
they are located in $\R^d$. So the costs of  $\langle x,y \rangle$
and  $\langle x+c,y+c \rangle$ are the same, however the cost of
$\langle x,y \rangle$ and $\langle y,x \rangle$ could be different. }. For a matching with multiple
edges, the cost of the matching is given by the sum of the cost of each
edge. We say an unmatched edge has infinite cost. We  refer to
such an $f$ as the cost function. Therefore, given a perfect matching
$M$
between red points given by
$R=\{x_i\}_{i=1}^n$ and blue points given by $B=\{y_j\}_{j=1}^n$, we call
$f[M]$ the cost of the matching where

  \begin{equation}
   \label{eq:costofamatching}
    f[M] = \sum_{i=1}^n f(M(x_i) - x_i).
  \end{equation}
%%%%%%%%%%%%%%
The above definition can be used to calculate the costs of different matchings
between finite collections of red and blue points with equal cardinality.

\begin{definition}[minimal matchings] \label{def:mm}

  \begin{enumerate}[(i)]
    \item  Given two finite collections of
points $R, B\in\R^d$ with equal cardinality, and a cost function
$f:\R^d\rightarrow\R$, we call a matching $M$ the $f$-minimal matching if
$f[M]\leq f[N]$ for any other matching $N$.

\item If  $R$ and $B$ have different cardinality we let the minimal matching
be given by the partial matching which minimizes the sum of cost of
the edges over partial matchings with as few unmatched points as
possible.

\item Given two countable collections of points $R,B\in\R^d$, and a
  cost function $f:\R^d\rightarrow \R$, we call $M$ an $f$-minimal
  matching if for any finite subset of the edges, the $f$-minimal
  matching between those points is given by $M$ restricted to those
  edges.
  \end{enumerate}
\end{definition}

Since imitating  \cref{def:mm}(i) for matchings
between countable collections of red and blue points results in the
cost of any given matching to be infinite, we use in (iii) a formulation which
arises from cyclical monotonicity in optimal transport  (see
\cite{villani2008optimal}). Intuitively, a matching between countable
collections of points in $\R^d$ is minimal if it cannot be locally
improved by a finite reordering of the edges.
In general, given a cost function $f$ and red and blue points $R,B
\in\R^d$, a minimal matching might not exist or even be unique. There may
even be infinitey many minimal matchings for certain cost functions
(See \cref{super}).

Below we define a property of cost functions called scale invariance which is a central notion in this paper.

\begin{definition}[Scale invariance]
Consider a cost function $f:\R^d\rightarrow\R$ and a collection of
points $R$ and $B$ as defined above such that there exists an $f$-minimal matching $M$ between
them. If for all $k>0$, an $f$-minimal matching between the rescaled
points $kR=\{kr:r\in R\}$ and $kB=\{kb:b\in B\}$ is given by $M_k$
where $M_k(kx)=kM(x)$ for all matched points $x\in R\cup B$, then we say that $f$ is
a scale invariant cost function.
\end{definition}

In other words,
zooming in or zooming out doesn't reorder a scale invariant minimal matching.
In \cite[Proposition 18]{minimal}, the authors classify all the scale
invariant cost functions on $\R$ which satisfy some regularity assumptions.
They assume that the cost function is symmetric, continuously
differentiable and strictly monotone. In \cref{thm:unbalcost} we
extend their result by assuming only measurability of the scale invariant cost
function.
In the theorem below, we view $\R^d$ in  spherical coordinates
$\R^d\cong\R_{\geq 0}\times S^{d-1}$.
\begin{samepage}
  \begin{theorem}[Any measurable scale invariant cost function]
  \label{thm:unbalcost}
 Let $R$ and $B$ be countable collections of red and blue points in $\R^d$. Any measurable  cost
function $f:\R_{\geq 0}\times S^{d-1}\rightarrow \R$ which results in
 a scale invariant minimal matching between $R$ and $B$ is given by
either
\begin{enumerate}
\item   \begin{equation} \label{eq:tmp1}
    f(r,\theta)=
      a(\theta)r^{\gamma}+b
    \end{equation}
    where $a:S^{d-1}\rightarrow\R$ is measurable, $b\in\R$, $\gamma\neq 0\in\R$ and
    $a\gamma>0$ or
  \item   \begin{equation}\label{eq:tmp2}
    f(r,\theta) = a \log(r)+b(\theta)
  \end{equation}
where $b:S^{d-1}\rightarrow(0,\infty)$ is measurable and $a>0$.
\end{enumerate}
\end{theorem}
\end{samepage}

 \cref{thm:unbalcost} shows that there cannot be a measurable and highly irregular
 discontinuous function that is also scale invariant. One way the cost
 function above is more  irregular compared to that in \cite{minimal}
 is via $a(\theta)$ and $b(\theta)$. If these two functions are
 non-constant, there may be edges of the same length which
 have different cost based on their orientation in space. We refer to
 this as the ``unbalancedness'' of the cost function. We present the proof of \cref{thm:unbalcost} in \cref{sec:cost}.

\begin{remark}[rotational invariance]
  In addition to scale invariance, if we were
  to also assume rotational invariance of the minimal matchings, we would
  recover the cost functions found in \cite{minimal}. For these costs,
  the functions $a(\theta)$ and $b(\theta)$ from \cref{thm:unbalcost}
  would be constants, and the cost of an edge would
  depend only on its length
  and not its orientation.\end{remark}

Notice that upon rescaling or translating a cost function, the minimal
matching induced by it remains unchanged. Therefore we say that two
cost functions $f$ and $g$ are equivalent to each other if there exist
constants $c_1,c_2\in(0,\infty)$ such that $f = c_1 g + c_2 $ or vice versa.
In this paper, we
specifically study what minimal matchings which arise from the above
cost functions look like in one dimension. In
\cref{cor:unbalcost_onedim} below, we list out
measurable scale invariant cost function in one dimension.

\begin{samepage}
  \begin{corollary}[$d=1$]\label{cor:unbalcost_onedim}
  A measurable scale invariant cost function in one dimension
  $f:\R\rightarrow\R$ must be equivalent to either
  \begin{align}
    \label{eq:costonedim}
    f_{\gamma,a}(x) &=
    \begin{cases}
      a|x|^\gamma, & x<0,\\
      x^\gamma, & x\geq 0,\\
    \end{cases} \qquad\qquad \mbox {or} \\
    f_{0,a}(x) &=
           \begin{cases}
             \log (a|x|) , & x<0,\\
             \log(x)  , & x\geq 0,
           \end{cases} \qquad \mbox{or}\nonumber\\
   f_{-\gamma,a}(x)&=
           \begin{cases}
             -a |x|^{-\gamma} ,& x<0,\nonumber\\             -
             x^{-\gamma}, & x\geq 0,
           \end{cases}
  \end{align}
  for some $a,\gamma>0$.
\end{corollary}
\end{samepage}

\begin{remark}
Note that for $f_{-\gamma,a}$, the cost function has
a negative coefficient to ensure that it will be an increasing
function of the distance. \end{remark}

When working in $\R$, we call edges $\langle r,b \rangle$ with $r<b$, red-blue edges and for
$b<r$, blue-red edges. The cost function introduced
in this paper in
\cref{cor:unbalcost_onedim} causes blue-red edges to be $a$ times as
expensive as red-blue edges of the same length. The parameter $a\in(0,\infty)$
represents a measure of bias against blue-red edges.

Naturally, $(\gamma,a)$-minimal matchings for the same set of points
but for different values of $\gamma$ and $a$ might lead to different
matchings. We demonstrate this in \cref{fig:sims} for the same $20$
red and $20$ blue points by plotting the different minimal matchings.  In \cref{fig:scatter}, we
consider the same $100$ red and $100$ blue points and plot the
distributions of the different minimal matchings.% \note{I could insert  \cref{fig:scatter} here. Should I do it? This would cause me to  separate \cref{fig:scatter} and \cref{fig:diffs}. }

In \cite{minimal}, the authors assume symmetry of the cost function
(which is not necessarily required for scale invariance) and study the
special case of \cref{cor:unbalcost_onedim} when $a=1$. We call the
corresponding cost function $f_{\gamma,1}$ the balanced
cost and the $f_{\gamma,1}$-minimal matching, the balanced $\gamma$-minimal
matching. When $a\neq 1$, we call $f_{\gamma,a}$ the unbalanced cost
and the corresponding $f_{\gamma,a}$-minimal matching, the unbalanced
$(\gamma,a)$-minimal matching. In \cite[Theorem 1]{minimal}, the
authors prove that for any value of $\gamma$, these minimal matchings
exist and are perfect. This proof verbatim can also be used to show
the same for unbalanced $(\gamma,a)$-minimal matchings.

  Our next main result is \cref{compwbal}, where we compare balanced
  and unbalanced minimal matchings.  Before we state this result, we
  discuss below some limiting behaviours
of the matching as $\gamma$ approaches $+\infty,-\infty, 0 $ and $1$. These
have been introduced in \cite{minimal}. Assume for now $a$ is fixed.

  \begin{enumerate}
  \item   To interpret $\gamma =0$, we take the limit as $\gamma\rightarrow
    0^+$, as $a=\alpha^\gamma$ for some $\alpha$ fixed.
Since adding or multiplying the cost function by a constant doesn't
    change the matching and $\lim_{\gamma\rightarrow 0^+}\frac{|\alpha
    x|^\gamma
  -1}{\gamma} = \log( \alpha x)$,  we consider the
$0$-minimal matching to be given by the cost function $\log(|\cdot|)$.

\item
  If we let $\gamma\rightarrow-\infty$, the cost of the
  minimal matching between $n$ pairs of points at $\{x_i\}_{i=1}^n$
  and $\{y_j\}_{j=1}^n$ is given by
\begin{equation}\label{eq:costofnpairs}
\min_{\sigma\in S_n} \sum_{i=1}^n - \left[
  a\mathbbm{1}_{\{x_i<y_{\sigma(i)}\}} +
  \mathbbm{1}_{\{x_i>y_{\sigma(i)}\}} \right]
(|x_i-y_{\sigma(i)}|)^\gamma .\end{equation}
As $\gamma\rightarrow -\infty$, the sum will be dominated by the
smallest edge regardless of whether its coefficient is $a$ or $1$.
Therefore, this matching aims to minimize the smallest edge.
  In   \cite[Lemma 8]{minimal}, the authors prove that the above construction
  of the
  $(-\infty)$-minimal matching results in a stable matching \cite{galeshapley}.
  A matching is said to be stable (or greedy) if for any two points $x,y$ of opposite color,
  we have \[|x-M(x)|\wedge|y-M(y)|\leq |x-y|.\]

\item If we take the limit as $\gamma\rightarrow\infty$, the cost in
 \eqref{eq:costofnpairs} will be minimized by the
 matching which minimizes the longest edge, since the sum is
 dominated by the longest edge. Again, for $a$ constant, this will
 not depend on $a$, but instead on the term
 $(|x-y_{\sigma(i)}|)^\gamma$ since we are taking
 $\gamma\rightarrow\infty$.
This is equivalent to the ``altruistic'' matching defined in
 \cite{minimal}.

 \item For $\gamma=1$, if we consider a configuration of four
 points as in \cref{fig:config}, with two red points followed by two
 blue points with distances $x$, $y$ and $z$ between consecutive
 pairs,  both the possible matchings
 have the same cost, $x+2y+z$. Since two independent Poisson point process can
 produce such a configuration of points infinitely often, we can
 switch between the two possible matchings to get uncountably many
 $(1,a)$-minimal matchings with locally the same cost.
 Therefore, we consider the limits of the
 cost function  as $\gamma\rightarrow 1^+$ and $\gamma\rightarrow 1^-$
 and call them the $(1^+,a)$ and $(1^-,a)$-minimal matchings
 respectively. For the configuration in \cref{fig:config}, the
 $(1^-,a)$-minimal matching will always contain the nested edges and the
 $(1^+,a)$-minimal matching will always contain the entwined edges.

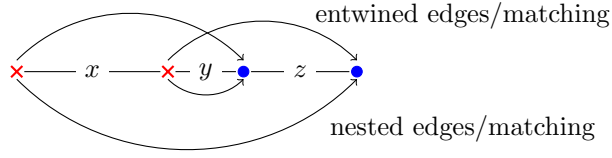
\begin{figure}[h]
  \centering
\begin{tikzpicture}
%  \filldraw [red] (0,0) circle (2pt);
        \draw[mark=x, color=red,thick,mark options={scale=1.5}] plot
  coordinates {(0,0)};
%  \filldraw [red] (2,0) circle (2pt);
        \draw[mark=x, color=red,thick,mark options={scale=1.5}] plot
  coordinates {(2,0)};
    \filldraw [blue] (3,0) circle (2pt);
    \filldraw [blue] (4.5,0) circle (2pt);
    \draw[-] (0+0.1,0) -- (2-0.1,0) node[midway,fill=white]{$x$} ;
    \draw[-] (2+0.1,0) -- (3-0.1,0) node[midway,fill=white]{$y$} ;
    \draw[-] (3+0.1,0) -- (4.5-0.1,0) node[midway,fill=white]{$z$} ;
    \draw[->] (0,0.15) [out=45,in=135] to (3,0.15);
    \draw[->] (2,0.15) [out=45,in=135] to (4.5,0.15);
    \draw[->] (0,-0.1) [out = 315,in=225] to (4.5,-0.1);
    \draw[->] (2,-0.1) [out=315,in=225] to (3,-0.1);
    \node at (5.9,0.75) [align=left] {entwined edges/matching};
    \node at (5.9,-0.75) [align=left] {nested edges/matching};
   \end{tikzpicture}
  \caption{The matching marked above is called the entwined matching
  and the one below is called the nested matching. For $\gamma>1$, the
  minimal matching includes the entwined edges and for $\gamma<1$, the
  minimal matching includes the nested edges.}\label{fig:config}
  \end{figure}

\end{enumerate}

We are almost ready to state \cref{compwbal}. Before we do that, we introduce the
following definition, which can be used to compare two matchings.

\begin{definition}[Finite difference]
  Two matchings $M$ and $M'$ between $R$ and $B$ are said to have
  finite difference if the multigraph formed by the union of their
  edges doesn't have an infinite component.
\end{definition}

In other words, two matchings with finite difference only differ on sets of finite size. There could be infinitely many such sets, however each chain
of reordered edges only affects finitely many points. Note that finite
difference is not a transitive property (we will see this in
\cref{rem:trans}). Finally, we are
ready to state \cref{compwbal}, where we compare the
unbalanced matchings introduced in this paper to the balanced minimal
matchings from \cite{minimal}.

\begin{samepage}
  \begin{theorem}[Comparison with balanced case]\label{compwbal}
  Let $R$ and $B$ be independent Poisson point processes with
  intensity $1$ on $\R$. Let $1\neq a\in\R$ be fixed.
  \begin{enumerate}[(i)]
  \item For $\gamma\in\{-\infty,1^-,1^+,\infty\}\cup(1,\infty)$, the set of balanced
    $\gamma$-minimal matchings and unbalanced $(\gamma,a)$-minimal
    matchings are identical.
  \item For $\gamma\in(-\infty,1)$ the unique balanced
    $\gamma$-minimal matchings and the unique unbalanced
    $(\gamma,a)$-minimal matchings have finite difference.
  \end{enumerate}
\end{theorem}
\end{samepage}

The proof of \cref{compwbal} can be found in \cref{sec:compwbal}. A
small variation of the above result, which combines it with
\cite[Theorem 3]{minimal} is listed below. It can be used to compare $(\gamma,a)$-minimal matchings
and $(\gamma',a')$-minimal matchings for any
$(\gamma,a)\neq(\gamma',a')$.

\begin{figure}[h]
  \centering
  \includegraphics[width=\textwidth,trim = {4cm 0 4cm 0},clip]{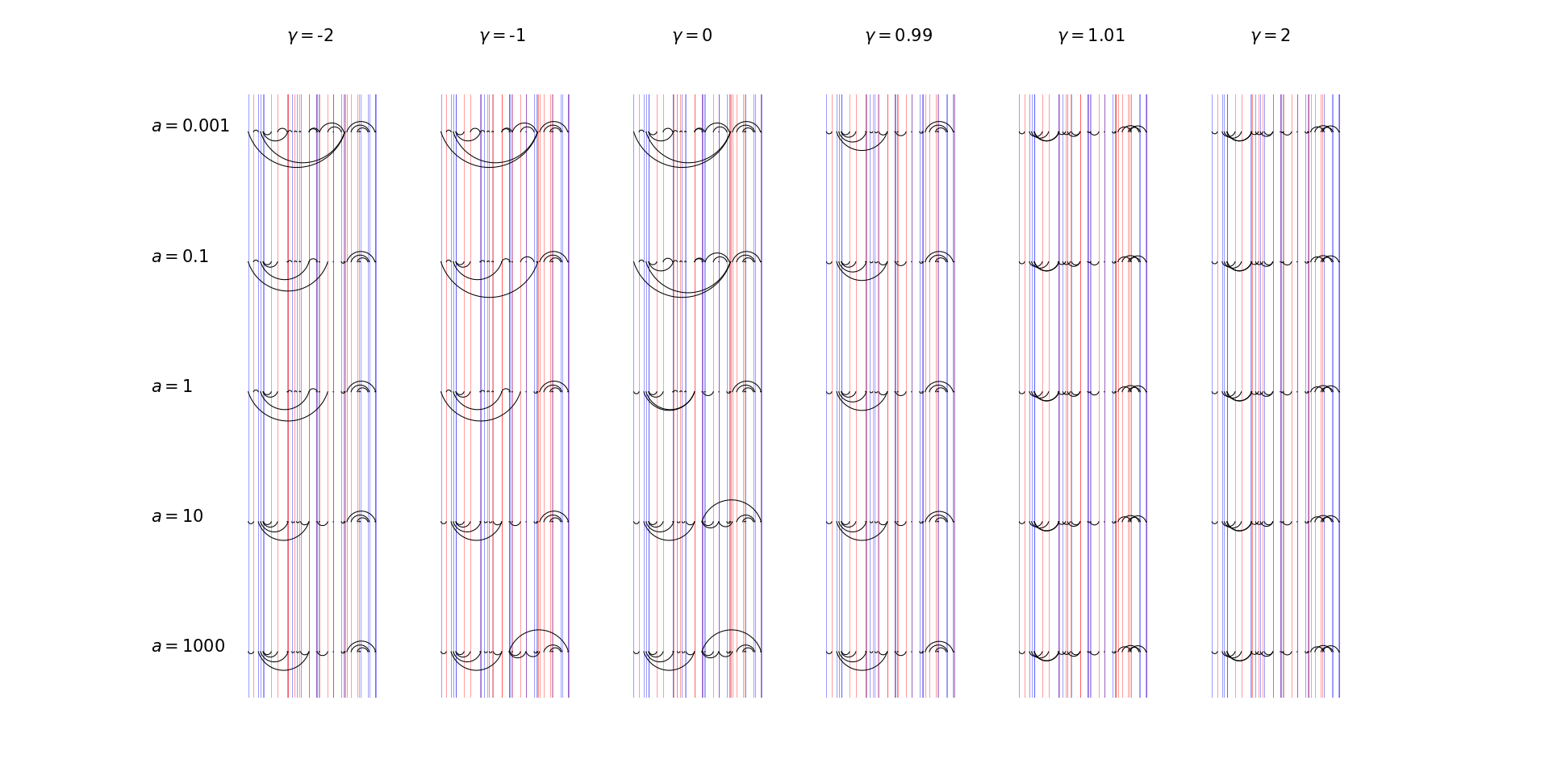}
  \caption{The $(\gamma,a)$-minimal matchings for different values of
  $\gamma$ and $a$ for the same $20$ red points and
  $20$ blue points, represented by the vertical lines. In the graph,
  blue-red edges have been drawn as downward arcs and red-blue edges
  have been drawn as upward arcs. }
  \label{fig:sims}
\end{figure}

\begin{samepage}
\begin{corollary}[Comparing any two different pairs of
  parameters]\label{cor:big}
  Let $R$ and $B$ be independent Poisson point processes with
  intensity $1$ on $\R$.
  \begin{enumerate}[(i)]
    \item If $\gamma,\gamma'\in(1,\infty]\cup\{1^+\}$ and $a,a'\in(0,\infty)$,
    the set of $(\gamma,a)$-minimal matchings is identical to the set
    of $(\gamma',a')$-minimal matchings.
    \item For $\gamma\in\{-\infty,1^-,\infty\}$, the set of $(\gamma,a)$ minimal matchings is identical to the
    set of $(\gamma,a')$-minimal matchings for any $a\neq
    a'\in(0,\infty)$.
    \item If $\gamma\in[-\infty,1)$, $\gamma'\in[-\infty,1)\cup\{1^-\}$ and
    $a,a'\in(0,\infty)$, then  the
    unique $(\gamma,a)$-minimal matching and the unique
    $(\gamma',a')$-minimal matching if $\gamma'\neq 1^-$ or any one of
    the countable $(1^-,a)$-minimal matchings if $\gamma'=1^-$ have finite difference.
  \end{enumerate}
\end{corollary}
\end{samepage}

 \cref{fig:sims} pictures $(\gamma,a)$-minimal
matchings for different values of $\gamma$ and $a$ for the same $20$ red
and $20$ blue points. Observe that all
minimal matchings with
$\gamma>1$ are identical and do not depend on $a$. \cref{cor:big}
suggests that the $(1^-,a)$-minimal matching doesn't
depend on $a$ and in \cref{fig:sims} we see this for $\gamma=0.99$.
Finally, we see that for subcritical values of $\gamma$, we get
different matchings depending on the parameters $a$ and $\gamma$.
Furthermore, we can observe that for these subcritical matchings,
increasing the value of $a$ leads to more red-blue edges since
blue-red edges are now more expensive and vice versa for smaller
values of $a$. In \cref{fig:cluster}, we will see distributions of
$(\gamma,a)$ which result in the same minimal matchings.

So far, we have assumed that $a$, the coefficient for the bias against
blue-red edges, is a finite positive constant. We get a few more interesting cases when we
consider the following limits of $a$, which have not been studied in
\cite{minimal}.

\begin{enumerate}
\item Suppose $a=\infty$ for a fixed finite $\gamma$. We assume that any blue-red edge has infinite cost and only
  consider matchings consisting entirely of red-blue edges. The
  behaviour of $(\gamma,\infty)$-minimal matching for different values of $\gamma$ is discussed in
  \cref{thm:a-infty}.
  \item Suppose $a=0$ for a fixed finite $\gamma$. We assume that every blue-red edge has
    zero cost. We discuss $(\gamma,0)$-minimal matchings in  \cref{thm:a-zero}.
    \item If we jointly let $a\rightarrow\infty$ or $a\rightarrow 0$ while
     $\gamma\rightarrow\infty$ or $\gamma\rightarrow-\infty$. This is the most interesting setup,
    since we get a new kind of matching not discussed before. We
     discuss this in \cref{thm:weirdlim}.
  \end{enumerate}

  Before we discuss any of the above results, we introduce some
  properties of matchings.

\begin{definition}[invariance]
  For $R,B$ countable subsets of $\mathbb{R}$ and $M$ a matching
  between them, if the law of $(R,B,M)$ is invariant under
  translations of $\mathbb{R}$, we say that the matching scheme $M$ is
  invariant.
\end{definition}

The invariance property of matchings is useful because this allows us to use
various ergodicity arguments.

\begin{definition}[factor]
  If a matching scheme $M$ is invariant and can be expressed as a
  deterministic function of $R$ and $B$ we call it a factor.
\end{definition}

In other words, a factor matching doesn't require any additional
randomness in its construction. In \cref{super}, we see examples
of matchings which are not factor matchings and require an arbitrary
edge in their construction.

\begin{definition}[locally infinite] \label{def:locinf}
We say that an edge $\langle x,y \rangle$ crosses a point
$z\in\mathbb{R}$ if $x<z<y$ or $y<z<x$. If every point $z\in\mathbb{R}$ is
crossed by finitely many edges in
a matching, then the matching is said to be locally finite. Otherwise
the matching is locally infinite.
\end{definition}

While \cref{def:locinf}   is specifically for $d=1$, it can be
adapted to higher dimensions as well.

Now we are ready to discuss the limits of $a$ in more detail, starting
with  $a=\infty$. For $\gamma>0$, this is equivalent to minimizing over matchings
which only contain red-blue edges, since setting $a=\infty$ here
implies that all the blue-red edges have infinite cost.

\begin{samepage}
\begin{theorem}[Taking $a=\infty$]\label{thm:a-infty}
  Let $R$ and $B$ be independent Poisson point processes of intensity
  $1$ on $\R$.
  The $(\gamma,\infty)$-minimal matchings are as follows.
  \begin{enumerate}[(i)]
  \item For $\gamma\in (-\infty,0)$, there exist uncountably many
    $(\gamma,\infty)$-minimal matchings.
  \item For $\gamma\in[0,1)\cup\{1^-\}
    $ and $a=\infty$, the cost of any given
    blue-red
    edge is infinite and hence minimizing over matchings with only
    red-blue edges, there exists a unique
    $(\gamma,\infty)$-minimal matching.
  \item For $\gamma=1$, similarly minimizing over matchings with only
  red-blue edges, there are uncountably many $(1,\infty)$-minimal matchings.
  \item For $\gamma\in\{1^+\}\cup(1,\infty)$, there does not exist such a
  $(\gamma,\infty)$-minimal matching consisting only of red-blue edges.
  \end{enumerate}
\end{theorem}
\end{samepage}

Next we discuss the case where $a=0$, where every blue-red
edge has zero cost. Unlike $a=\infty$, where we minimize over matchings
consisting entirely of red-blue edges, in \cref{thm:a-zero}, we minimize over all possible
matchings.

\begin{samepage}
\begin{theorem}[Taking a=0]\label{thm:a-zero}
 Let $R$ and $B$ be Poisson point processes of intensity $1$ on $\R$.
 If we consider the cost function in \cref{cor:unbalcost_onedim} with
 $a=0$, the following holds.
 \begin{enumerate}[(i)]
     \item For $\gamma\in(-\infty,0)$, there is a unique $(\gamma,0)$-minimal matching which is locally   infinite and a factor.
\item For $\gamma=0$ the cost function is no longer scale invariant.
 However uncountably many non-scale invariant minimal matchings may be constructed.
 \item For $\gamma\in(0,\infty)\cup\{1^+\}\cup\{1^-\}$
 there are uncountably many $(\gamma,0)$-minimal matchings.
 \end{enumerate}
\end{theorem}
\end{samepage}

\cref{thm:a-zero,thm:a-infty} have been proven in \cref{sec:lims}. In
\cref{rem:lims}, we discuss why these results may seem slightly
counterintuitive.

Below we discuss the most interesting of the limits in $a$, the joint
limits of $a\rightarrow\{0,\infty\}$ and
$\gamma\rightarrow\pm\infty$. Before we do this, we introduce two
definitions below. Let $\alpha\in(0,\infty)$ be fixed.

\begin{definition}[$\alpha$-stable matching]\label{def:alphastab}
  A matching between sets $R$ and $B$ in $\R$ are said to be
  $\alpha$-stable if for any two points $x\in R$ and $y\in B$ not
  matched to each other, we have
  \[
    c(x,M(x)) \wedge c(M(y),y) \leq c(x,y)
  \] where
  \begin{equation}
    c(x,y) = [\mathbbm{1}_{x\leq
  y}+\alpha\mathbbm{1}_{x>y}]|x-y| = \begin{cases} \alpha|x-y|, &
                                                                  x>y, \\ |x-y|, & x\leq y.\end{cases}
  \end{equation}
\end{definition}

Now, if $\alpha=1$, this matching is the same as the stable matching
defined in \cite{galeshapley,minimal}. However for $\alpha\neq 1$ it
is different from every other matching that we have discussed so far.
One can think of this matching to be constructed by the following
procedure. Let us assume that at a given red point at $r$, there is an
interval $I_r^t$ of size $(1+1/\alpha)t$ at time $t$ growing around it given by
$I_r^t=(r-t/\alpha ,r+t)$. This red point at $r$ is matched to the first unmatched
blue point to enter its interval $I_r^t$. This results in a greedy matching where the preferences of the points are
given by $c(x,y)$ which prefers to match in some directions $\alpha$
times more than
others.
In the next definition, we construct the ``altruistic'' equivalent of the
``greedy'' definition above.

\begin{definition}[$\alpha$-altruistic matching] \label{def:alphaalt}
A matching between sets $R$ and $B$ in $\R$ is defined to be
$\alpha$-altruistic if it minimizes in dictionary ordering, the
  sequences of $c(x,y)$ in descending order, where $c(x,y)$ is given
  by
    \[c(x,y) = [\mathbbm{1}_{x\leq
  y}+\alpha\mathbbm{1}_{x>y}]|x-y| = \begin{cases} \alpha|x-y|, &
  y<x, \\ |x-y|, & x\leq y.\end{cases} \]
\end{definition}

Here, we aim to minimize the largest possible value of the cost of the
edge $\langle x,y \rangle$ where the cost of the edge is given by
$\alpha|x-y|$ for a blue-red edge and $|x-y|$ for a red-blue edge.

\begin{theorem}[joint limit]\label{thm:weirdlim} If we consider the
  $f_{\gamma,a}$-minimal matchings defined in
  \cref{cor:unbalcost_onedim}, the following hold for fixed $\alpha>0$.
  \begin{enumerate}[(i)]
  \item The $(\gamma,\alpha^\gamma)$-minimal matching is given by the
  $\alpha$-stable matching as $\gamma\rightarrow-\infty$.
  \item The $(\gamma,\alpha^\gamma)$-minimal matching is given by the
  $\alpha$-altruistic matching as $\gamma\rightarrow\infty$.
  \end{enumerate}
\end{theorem}

\begin{remark}
Note that in one dimension the $\alpha$-altruistic matching is the
same as the altruistic matching. This need not be the case in higher
dimensions. On the other hand, the $\alpha$-stable matching for
$\alpha\neq 1$ behaves
differently from the greedy matchings in all dimensions, including
$d=1$. An intuitive interpretation of the $\alpha$-stable matching in
higher dimensions is
given in \cref{rem:alphastab}.
\end{remark}

In the above theorems, we have listed all the possibilities for limits
of  $\gamma$ and $a$. Note that if $a$ increases faster than
$\alpha^\gamma$, the minimal matching will resemble the situation in
\cref{thm:a-infty} where $a=\infty$ and if $a$ grows slower than every $\alpha^\gamma$
then the minimal matching will resemble the situation in
\cref{thm:a-zero} where $a=0$.

\subsection{Discussion} \label{ssec:discussion}

\begin{figure}[!htb]
  \centering
  \begin{subfigure}{\textwidth}
    \includegraphics[width=0.81\textwidth,trim={1cm 0.2cm 2cm 1cm},clip]{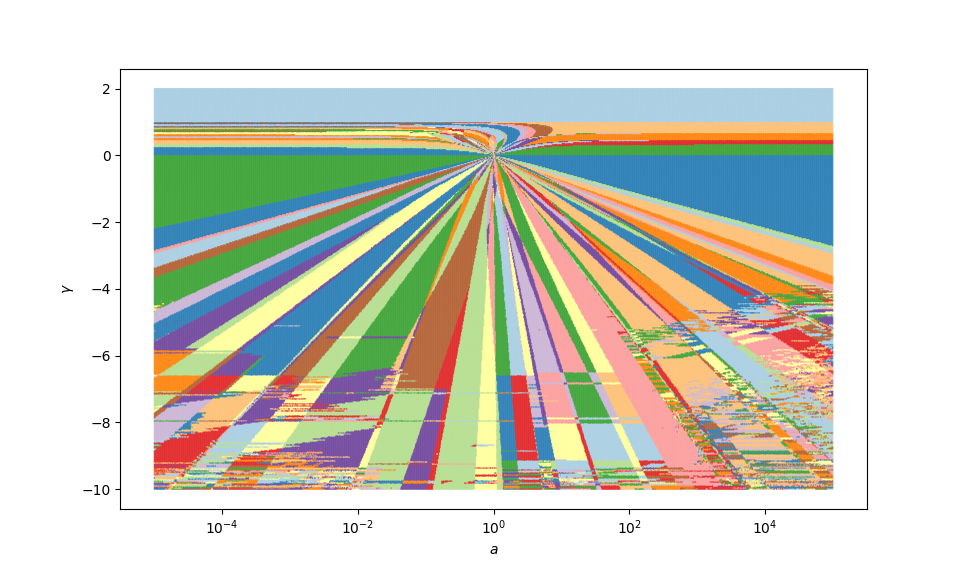}
    \caption{Scatterplot of $(\gamma,a)$-minimal
    matchings.}\label{fig:scatter}
  \end{subfigure}
    \begin{subfigure}{\textwidth}
    \includegraphics[width=\textwidth,height=6.3cm,trim={1.8cm 0.2cm 2.4cm
      1cm},clip]{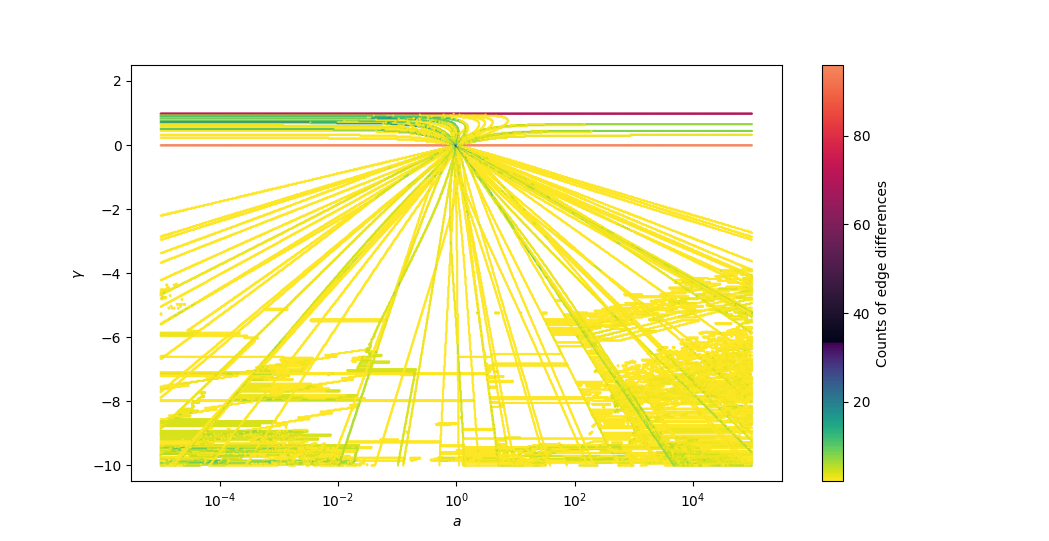}
    \caption{Counts of edge differences between minimal matchings.}\label{fig:diffs}
  \end{subfigure}
  \caption{ For the same $100$ red and $100$ blue points, we plot the
  distributions of the different $(\gamma,a)$-minimal matchings. In
 \cref{fig:scatter}, different squares of the same color denote the
  same matching for those values of $(\gamma,a)$. Note that $a$ has
  been plotted on the $x$-axis with a log scale.
  In \cref{fig:diffs}, we plot the counts of edge differences in the
  $(\gamma,a)$-minimal matchings for adjacent values of $\gamma$ and
  $a$.
  }\label{fig:cluster}
  \end{figure}

  \begin{figure}[!htb]
    \centering
    \begin{subfigure}{0.9\textwidth}
      \centering
      \includegraphics[width=\textwidth,trim={5cm 2cm 5cm 2.8cm},clip]{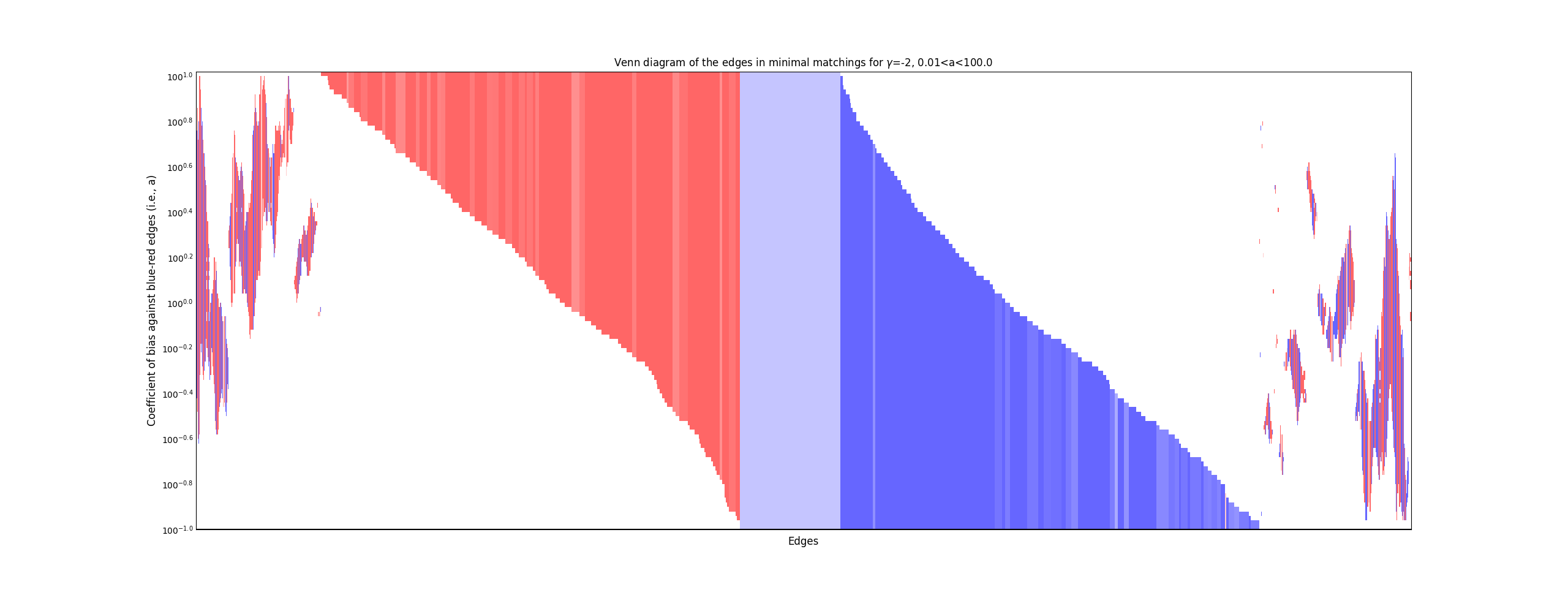}
      \caption{Venn diagram of edges for $\gamma=-2$ and $0.01<a<100$.}
\end{subfigure}

\begin{subfigure}{0.9\textwidth}
\includegraphics[width=\textwidth,trim={5cm 2cm 5cm 2.9cm},clip]{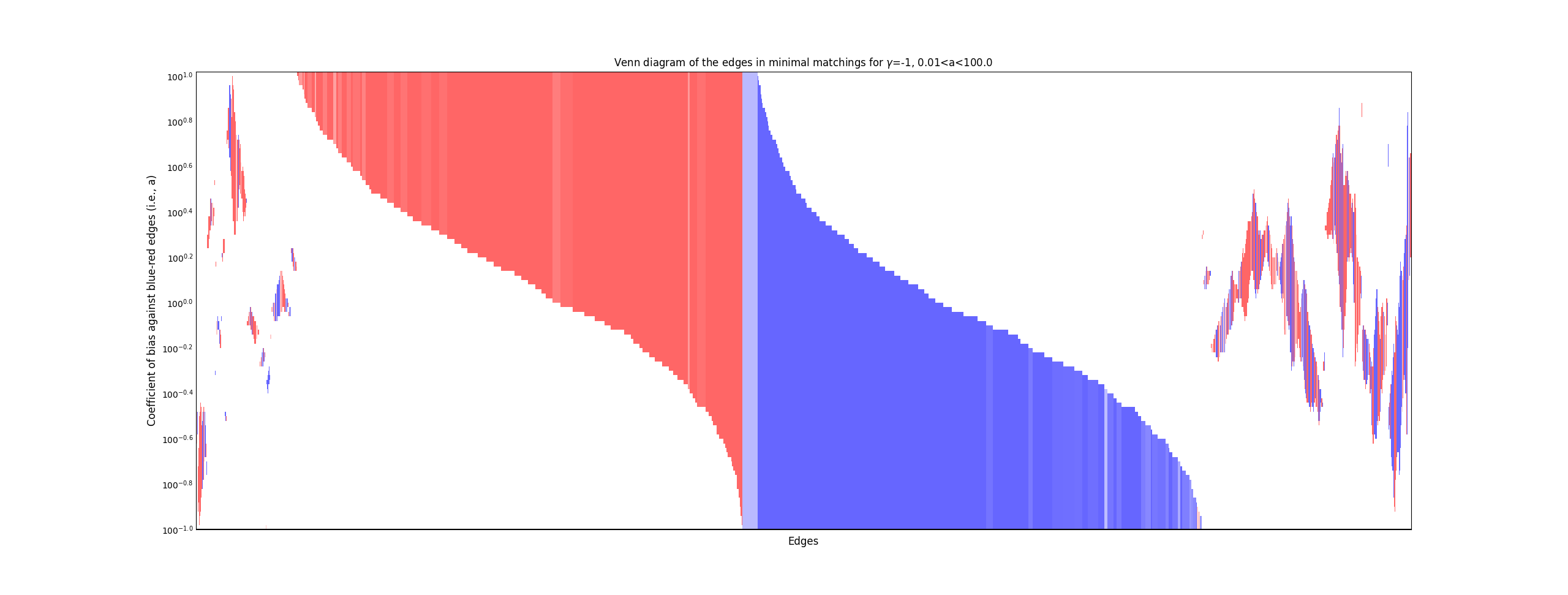}
    \caption{Venn diagram for $\gamma=-1$ and $0.01<a<100$.}
\end{subfigure}

\begin{subfigure}{0.9\textwidth}
\includegraphics[width=\textwidth,trim={5cm 2cm 5cm 2.8cm},clip]{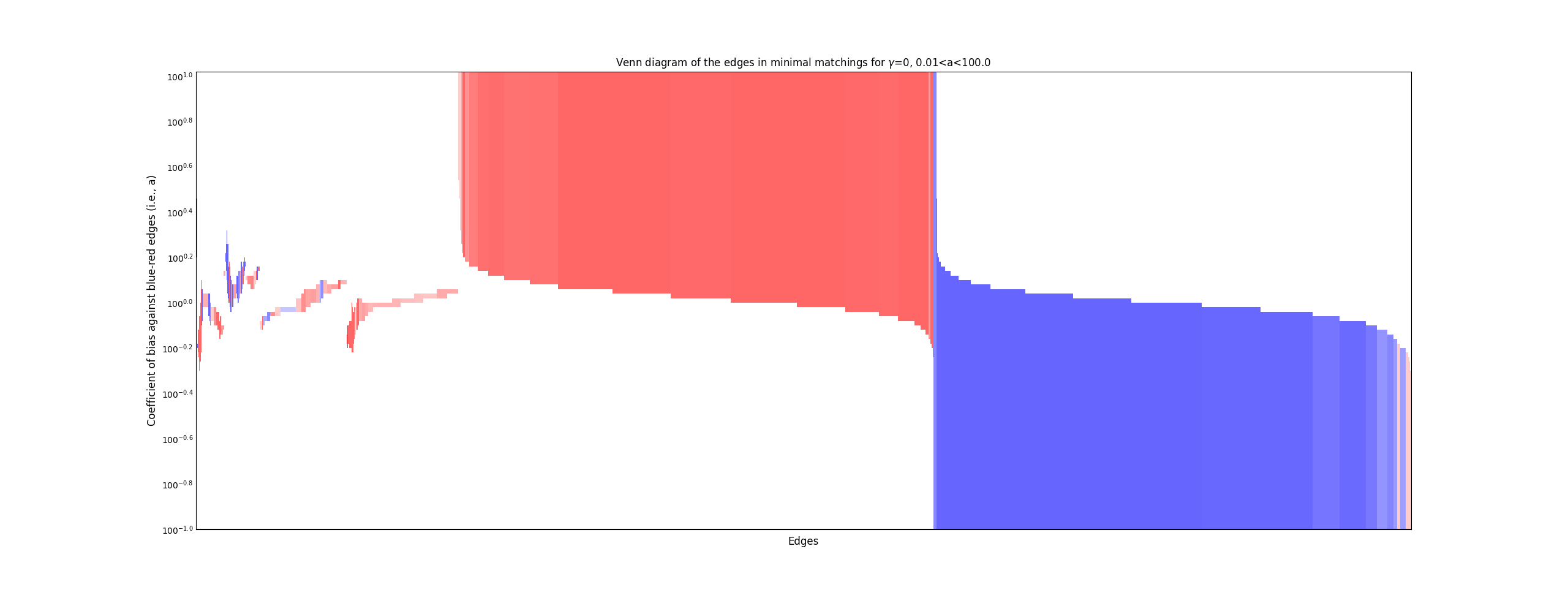}
    \caption{Venn diagram for $\gamma=0$ and $0.01<a<100$.}
\end{subfigure}
\caption{The above plots contain Venn diagrams of edges used in the $(\gamma,a)$-minimal
  matchings for the same $1000$ fixed red and blue points. Each column
  above represents a collection of the edges and the rows represent
  the minimal matchings for the different values of $a$. The columns
  have been colored based on the orientation of most of the edges in
  the columns. So the darker red columns consist entirely of red-blue
  edges and the lighter blue columns consist of slghtly more blue-red
  edges, but still many red-blue edges. It is easy to see that as we
  increase the value of $a$, for all values of $\gamma$, the matchings
  start to contain more red-blue edges. }\label{fig:vennbig}
\end{figure}

In this section, we briefly discuss some observations from our
results.

In \cref{fig:cluster}, we plot the $(\gamma,a)$-minimal matchings for
different values of the parameters of the cost function for the same $100$
red and $100$ blue points. For several values of
$(\gamma,a)$ with $\gamma\in(-10,2)$ and $a\in(10^{-5},10^{5})$. We
compute the $(\gamma,a)$-minimal matching and color the pixels to
ensure that each matchings is represented by a single color\footnote{Since
there are more than $900$ different minimal matchings, the color
palette we have used repeats colors for different matchings.}. It
can be seen that for supercritical $\gamma$, the minimal matching is
identical and it doesn't depend on $a$ or $\gamma$ as long as
$\gamma>1$. Even for subcritical values of $\gamma$, we
 observe that for several adjacent values
of $(\gamma,a)$, the minimal matchings are the same. When $a$ (on the
$x$-axis) is plotted on the $\log$ scale, matchings appear to be
distributed somewhat linearly in the parameters. This agrees with our
observation in \cref{thm:weirdlim}, where the $(\gamma,\alpha^\gamma)$-minimal
matching converges as $\gamma\rightarrow\infty$, since the lines in
\cref{fig:scatter}
correspond to values of $a$ such that $a=\alpha^\gamma$ for some
$\alpha>0$.

Since \cref{fig:scatter}, identifies different matchings but doesn't
classify how different these matchings are, in \cref{fig:diffs}, we
plot the number of edges differing between adjacent matchings in
\cref{fig:scatter}. Notice that most of the matchings we see with
$\gamma<1$, differ only by a few edges. Similarly for
$\gamma\in(0,1)$. This agrees with the construction of the $\alpha$
stable matchings in \cref{def:alphastab}, where we don't expect
slightly changing $\alpha$ to significantly alter the matching. The minimal
matchings significantly change at $\gamma=1$, which is
predicted in \cite[Theorem 2,3]{minimal}. The minimal matchings change
completely at $\gamma=0$ for values of $a$ away from $1$. This is
because of the argument in
\cref{sec:modelmainresults}, which tells us that it is not the
$(\gamma,a)$-minimal matchings, but instead the $(\gamma,a^\gamma)$-minimal
matchings which converge to the $(0,a)$-minimal matching as $\gamma\rightarrow
0$. Observe that the lines in \cref{fig:diffs} correspond to the
graphs of $u^\gamma + v^\gamma = (1+u+v)^\gamma + a$ for different
values of $u$ and $v$. These graphs represent the values of
$(\gamma,a)$ where the minimal matching changes exactly one pair of
edges (from nested edges to separate edges).

In \cref{fig:vennbig}, we consider the same $1000$ red and $1000$ blue points
and plot a venn diagram of the $(\gamma,a)$-minimal matchings. We
consider $\gamma\in\{-2,-1,0\}$ with a range of values of $a$ from 0.01 to 100 plotted on a
$\log$ scale. We see that for small values of $a$, the matchings
contain mostly blue-red edges and for larger values of a, they contain
mostly red-blue edges. We see that for the different values of
$\gamma$, the rates at which the blue-red edges are replaced on
increasing $a$ are different, and the fraction of edges used in all
the matchings are also different.
Furthermore, it appears that for most edges,
as we increase the value of $a$, once the edge stop being used, it is
not typically used again for a larger value of $a$. Similarly when we
reduce the value of $a$. For every edge, there is a set of the values
of $(\gamma,a)$ for which it is used in the minimal matching. We
believe that for most edge, this set is monotone in $a$ and in
$\gamma$.

\subsection{Outline of the paper}

The outline for the rest of the paper is as follows.

\begin{itemize}
\item In
\cref{sec:cost}, we provide a proof for \cref{thm:unbalcost}. The
proof is very analysis heavy and technical, the ideas used in this
proof are very different from the rest of the paper. In this section,
we show that any measurable
scale invariant cost function must be monotone (\cref{thm:monotone}),
continuous (\cref{thm:continuous}) and then introduce a multiplicative convolution
which is useful to show that these cost functions will be $C^1$. Then
we use some existing results to prove \cref{thm:unbalcost}, which
classifies all possible measurable scale invariant cost functions.

\item In \cref{sec:compwbal}, we provide a proof of \cref{compwbal}.
  The proof somewhat resembles the proof of
                        \cite[Theorem 3]{minimal}, which compares
                        balanced $\gamma$-minimal  matchings for two
                        different values of $\gamma.$
  A significant
  difference come from proving that the unbalanced
  $(\gamma,a)$-minimal matchings also satisfy  quasistability
  (\cref{quasi}).
 In \cref{rem:canonical}, we discuss some results from \cite{minimal}
  which also hold for unbalanced matchings.

\item
Finally in \cref{sec:lims}, we provide the proofs of
\cref{thm:a-infty,thm:a-zero,thm:weirdlim} which deal with the
limits of $a$ approaching $0$ or $\infty$.
\end{itemize}

\section{Cost Functions (Proof of \cref{thm:unbalcost}) }\label{sec:cost}

In \cite{minimal}, the authors consider matchings on $\mathbb{R}$,
assume a scale invariant cost function $f$ is symmetric,
continuously differentiable, non-decreasing and non-constant and then
show that there exist $a,b,\gamma\in\mathbb{R}$ such that
\[
  f(x) =
  \begin{cases}
    ax^\gamma +b, & \gamma\neq 0, \qquad \mbox{or} \\
    a\log x + b, & \gamma = 0.\\
  \end{cases}
\]
We extend this result in \cref{thm:unbalcost} to find all measurable  scale invariant cost
functions on $\mathbb{R}^d$.

The proof is split into several propositions. In \cref{thm:monotone},
we show that any scale invariant cost
function must be monotone and in \cref{thm:continuous}, we show it must
be continuous. Then, %we use a result from \cite{minimal} to conclude
%what the function should look like.
to use \cref{lem:sc-inv} (analogous to  \cite[Proposition
18]{minimal} but for the unbalanced cost), we also need to
show that this function is $C^1$, which we do by introducing a multiplicative
convolution and an approximate identity to recover the cost functions (in
\cref{convolution} and
\cref{approxid}).  Finally, we prove \cref{thm:unbalcost}.

We consider a red point at the origin
and  try
to deduce what a scale invariant cost function can look like along the positive $x$-axis.
A similar analysis can be used for any other ray $\{k \mathbf{v}:k>0\}$ along a given
vector $\mathbf{v}$.
Let $f:(0,\infty)\rightarrow\mathbb{R}$ be a
restriction of the cost function along the positive $x$-axis. In
\cref{thm:monotone}
 below, we show that the cost function must be monotone along rays.

\begin{proposition}[monotone]\label{thm:monotone}
  A measurable scale invariant cost function defined along a ray $f:(0,\infty)\rightarrow \mathbb{R}$ must be
  monotone.
\end{proposition}

\begin{figure}%\label{fig:monotone}
  \centering
  \begin{tikzpicture}
%    \filldraw [red] (0,0) circle (2pt);
          \draw[mark=x, color=red,thick,mark options={scale=1.5}] plot coordinates {(0,0)};
    \filldraw [blue] (2,0) circle (2pt);
    \filldraw [blue] (3,0) circle (2pt);
    \filldraw [blue] (4.5,0) circle (2pt);
    \draw[<->] (0,-1/4+0.07) -- (2,-1/4+0.07) node[midway,below]{$r$} ;
    \draw[<->] (0,-1/2-0.02) -- (3,-1/2-0.02) node[midway,below]{$1$} ;
    \draw[<->] (0,1/4) -- (4.5,1/4) node[midway,above]{$s$} ;
  \end{tikzpicture}
  \caption{A configuration of red and blue points used to show
  monotonicity.}\label{fig:monotone}
\end{figure}
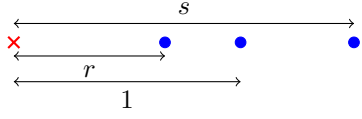

\begin{proof}[Proof]
  Let the set $R$ include one red point and the set $B$ include three
  blue points, distances $r<1<s$ from the red point along the same
  direction as seen in \cref{fig:monotone}.
  The minimal matching will have
  the edge with the smallest cost and two unmatched blue points. Let's
  assume our cost function isn't monotone and that the smallest out of
  the three costs
  $f(r), f(1)$ and $f(s)$ is $f(1)$. A similar, mirrored argument can
  be used to ensure we cannot have a local maxima at $f(1)$. Scale invariance of the cost
  function implies  \begin{equation*}
  \forall k\in(0,\infty),\qquad   f(k)<f(kr) \mbox{ and } f(k)<f(ks).
\end{equation*}

  Now, if $\log_s r<0$ is rational, say $-m/n$ where $m,n\in\N$, then we have
  $s^{-m} = r^n$. Due to scale invariance of $f$ we  have
  \begin{align*}
    f(1)<f(r)<\cdots<f(r^{n-1})<f(r^n)=f(s^{-m})
   % <f(s^{-m+1})
    <\cdots<f(s^{-1})<f(1),
  \end{align*}
  leading to a contradiction.

  Therefore $\log_s r$ must be irrational. This means that numbers of the
  form $\{r^n s^m\}_{n,m\in\mathbb{N}}$ are dense in $(0,\infty)$. We
  can show this by considering an interval $(k,l)\subset\mathbb{R}$
  and finding $m,n\in\mathbb{N}$ such that $$ \log_s k < m + n \log_s
  r < \log_s l. $$ This is always possible for positive $m$ and $n$
  because $n\log_s r$ is negative and the fractional part $\{n\log_s
  r\}_{n\in\mathbb{N}}$ is dense in $(0,1)$ since $\log_s  r$ is
  irrational. The above equation ensures that $k<r^n s^m<l$, so
  numbers of the form $\{r^m s^n\}_{m,n\in\mathbb{R}}$ are dense in
  $\mathbb{R}$.

  Now, we define the following three sets, which form a partition of
  $(0,\infty)$,
  \begin{align*}
    \Sigma_+ &= \{x\in(0,\infty):f(x)>f(1)\},\\
    \Sigma_-& = \{x\in(0,\infty):f(x)<f(1)\}, \\
    \Sigma &= \{x\in(0,\infty):f(x)=f(1)\}.
  \end{align*}
  These three sets must be measurable
  because $f$ is measurable.

  Let $m$ denote the Lebesgue measure. For a measurable set $A$, the
  Lebesgue density theorem \cite{realanalysis_steinshakarchi} defines
  the Lebesgue density of a set $A$ around a point $x$, given by
  $D_A(x)$ in the following way
  \begin{equation}
    \label{eq:ldt}
    D_{A }(x) := \lim_{\epsilon\rightarrow 0} \frac{m(A\cap
    B(x,\epsilon))}{m(B(x,\epsilon))} =  0 \mbox{ or } 1 \mbox{ a.e.
    }x\in\mathbb{R},
  \end{equation}
  and states that $m(A\Delta D(A))=0$, where $D(A) = \{x:D_A(x) =1\} $ denotes the
  set of density points for the set $A$. Intuitively, \cref{eq:ldt} tells us that
  for a measurable set, almost all the points are either  ``inside''
  the set, i.e., $D_A(x)=1$, or
  ``outside'' the set, i.e., $D_A(x)=0$. Furthermore, the set of points ``inside'' the set almost
  coincide with the set itself.
  To proceed with the proof, we consider the following three cases
  separately. Either the set ${\Sigma_+}$ has a density point, or
  the set
  ${\Sigma_-}$ has a density point, or neither of them has any
  density points.

  {\bf Case 1:} Let the set ${\Sigma_+}$ have a density point given
  by $x_0$, i.e., $D_{\Sigma_+}
  (x_0)=1$.
  We show that the point $rx_0$ must also be a density
  point of ${\Sigma_+}$ below.
\begin{leftbar}
  \begin{claim}\label{cl1}
    If $D_{\Sigma_+}(x_0)=1$ then $D_{\Sigma_+}(rx_0)=1$.
  \end{claim}
  \begin{proof}
    If $x\in\Sigma_+\cap B(x_0,\epsilon)$ then $rx\in
    B(rx_0,r\epsilon)$ and $rx\in\Sigma_+$ since $f(rx)>f(x)>f(1)$.
    This implies
    \begin{align*}
      r \left( \Sigma_+ \cap B(x_0,\epsilon) \right) & \subseteq \Sigma_+\cap
      B(rx_0,r\epsilon)  \\
      \implies r m(\Sigma_+ \cap B(x_0,\epsilon)) & \leq m (\Sigma_+\cap
      B(rx_0,r\epsilon)).
    \end{align*}
    Therefore,
    \begin{equation*}
      \label{eq:1}
     1\geq \lim_{\epsilon\rightarrow 0} \frac{m( \Sigma_+\cap
      B(rx_0,r\epsilon))}{m(B(rx_0,r\epsilon))} \geq
      \lim_{\epsilon\rightarrow 0} \frac{rm(\Sigma_+\cap
      B(x_0,\epsilon)}{rm(B(x_0,\epsilon)} = 1,
    \end{equation*}
   proving \cref{cl1}.\end{proof}
\end{leftbar}
  Similarly we can show that $sx_0$ is also a density point for
  ${\Sigma_+}$, i.e., $D_{\Sigma_+}(sx_0)=1$. Furthermore, the above  argument can be applied iteratively to show that
 for any $m,n\in\N$, we have $D_{\Sigma_+}(r^n s^m x_0)=1$ implying that
 the set of points given by
 $\{r^n s^m x_0:n,m\in\N\}$  are all density points of $\Sigma_+$.
  This gives us a dense collection of density points of the set
  $\Sigma_+$.
  In \cref{cl2} below, we show that given a dense collection of density points of the
  set ${\Sigma_+}$, it must be true that every point in $(0,\infty)$ must be
  a density point of $\Sigma_+$.
 \begin{leftbar}
 \begin{claim}\label{cl2}
    Consider any $y\in(0,\infty)$, we have $D_{\Sigma_+}(y)=1$.
  \end{claim}
  \begin{proof}
    Arbitrarily choose $\tilde\epsilon>0$ and $\delta>0$ small.

    Consider the ball $B(y,\tilde\epsilon)$. Let
    $\{ r^{n_k}s^{m_k}x_0\}_{k\in\N} $ be dense in $B(y,\te)$. For
    simplicity, let $x_k = r^{n_k} s^{m_k}x_0$.    For each $k\in\N$, let
    \begin{align}
      \label{eq:epsilons}
      \e_k = \sup\{\e: & \quad B(x_k,\e)\subseteq B(y,\te) \qquad \mbox{ and
      }\nonumber \\
     & \quad 1-\delta < \frac{m(\Sigma_+\cap
      B(x_k,\e'))}{m(B(x_k,\e'))} <1,  \qquad   \forall \e'<\e \}.
    \end{align}
In \eqref{eq:epsilons}, for each $x_k$, we choose an $\e_k$ so that the ball
    $B(x_k,\e_k)$ lies in $B(y,\tilde \e)$ and ``most'' points in
    $B(x_k,\e_k)$ are in $\Sigma_+$.
    Notice that $\{ B(x_k,\e_k)\}_{k\in\N} \subseteq B(y,\te)$ is a cover of
    $\overline{B(y,\te (1-\delta))}$ which is a compact set. Consider
    a finite subcover that is also minimal (so that each point is
    covered by at most two balls)\footnote{ we can do this because any ray $\{k
    \mathbf{v}:k>0\}$ is in $\R$.}, given by
    $\{B(x_{k_i},\epsilon_{k_i})\}_{i=1}^n$. For simplicity, let $B_i
    = B(x_{k_i},\epsilon_{k_i}) .$
    From our choice of $\e_k$, we know that for all $i$, most of the
    points in $B_i$ are in $\Sigma_+$ so there must be few in its
    complement. That is,
    \begin{align*}
      m(\Sigma_+\cap B_i)&\geq (1-\delta) m(B_i)\\
    \implies  m(\Sigma_+^C\cap B_i)&\leq \delta m(B_i).
    \end{align*}
This along with union bound and the minimality of the cover implies
     $$ m(\Sigma_+^C \cap (\cup_{i=1}^n B_i))\leq
    \delta\sum_{i=1}^n m(B_i)\leq \delta 2 m(B(y,\te (1-\delta))).$$
Therefore, we have \[\frac{m(\Sigma_+^C \cap B(y,\te
        (1-\delta)))}{m(B(y,\te (1-\delta)))} \leq 2\delta.\]    Since $\delta$ was arbitrarily chosen, we can let $\e \rightarrow
        0$ to get $D_{\Sigma_+^C}(y) =0$ implying
    that $y$ must be a density point of $\Sigma_+$, proving the claim.\end{proof}

\end{leftbar}

Now we have shown that if there is one density point $x_0$ for the set
    $\Sigma_+$,  every point  $y\in(0,\infty)$ must also be a density point of
    $\Sigma_+$. % Therefore, $(0,\infty)\in D(\Sigma_+)$.
    From Lebesgue Density Theorem we also know that the
    symmetric difference of $\Sigma_+$ and
    $D(\Sigma_+)$, must have measure zero, so we must have
    $y\in\Sigma_+$ for a.e. $y\in(0,\infty)$.
    This gives us a contradiction because if a.e. $x\in(1,2)$ had
    $x\in\Sigma_+$, we would have a.e. $x\in(1/2,1)$ have
    $x\in\Sigma_-$. Since $f(x)>f(1)\implies f(1)>f(1/x)$ due to
    scale invariance and because locally Lipschitz functions (like
    $x\rightarrow 1/x$ on $(0,\infty)$) map null sets to null sets. %continuous and differentiable
    %functions like $1/x$ preserve sets of measure zero.

    {\bf Case 2:} There exists $x_0\in\R$ such that
    $D_{\Sigma_-}(x_0)=1$.

    We follow a similar proof as above and show $\{r^{-n}
    s^{-m}x_0\}_{m,n\in\N}$ are dense in $(0,\infty)$ and consequently
    $y\in\Sigma_-$ for a.e. $y\in(0,\infty)$ resulting in a contradiction.

    {\bf Case 3:} Neither $\Sigma_+$ nor $\Sigma_-$ have any density
    points. Due to Lebesgue density theorem, these two sets must have
    measure zero.  So a.e. $y\in(0,\infty)$ must be
    contained in the set $\Sigma=\{x\in(0,\infty):f(x)=f(1)\}$,
    implying that $f$ is a constant almost everywhere.
     Now, $r\Sigma = \{rx:x\in\Sigma\} $ is a subset of $ \Sigma_+$ since
    $f(rx)>f(x)=f(1)$ which means $m(\Sigma_+)\geq r m(\Sigma)>0$
    giving us a contradiction.
\end{proof}

%%%%%%%%%%%%%%%%%%%%%%%%%%%%%%%%%%%%%%%%%%%%%

%The following remark contains an example of a non-measurable function
%which is scale invariant and non-monotone.

\begin{remark}
  It is possible to construct a cost function that is not monotone and
  not measureable, there is a nice example in \cite[Lemma
  2.2]{toure2021scaling} using a non-linear solution of Cauchy's functional
  equation, using the Hamel basis \cite{hamel1905basis}.
\end{remark}

Above in \cref{thm:monotone}, we established that the function must be monotone along all
rays. If the function is increasing along one ray and decreasing
along another, we can use a simple scale invariance argument
comparing two edges with the same cost. Upon rescaling, one of these would
increase and the other would decrease, contradicting scale invariance.
This ensures that the cost functions along different rays must either
all be monotonically
increasing or monotonically decreasing as we move away from zero.
Now, if our cost function was a decreasing function of the distance,
i.e., longer edges are cheaper than shorter edges, the minimal
matching would be well defined for finite collections of points, but
not for infinitely many points such as Poisson point processes on $\R^d$. In the proof
of \cite[Theorem 1]{minimal}, which proves that these $\gamma$-minimal
matchings exist and are perfect, we see that a crucial property of Poisson point
processes that allows us to have local minima for increasing cost functions is
that a Poisson point process on $\R^d$, cannot have an infinite
decreasing sequence of edges. However, a Poisson point process does in fact
have an infinite sequence of increasing edges, this means there can
never be a minimal matching for a cost function which is decreasing
with the distance, since by flipping along any infinite chain of increasing
edges, we could always be improving the cost.
Due to this, we consider cost functions which are increasing functions
of the distance along each ray from now on.

%%%%%%%%%%%%%%%%%%%%%%%%%%%%%%%%%%%%%%%%%%%%%

We go back to our analysis in one dimension. We have shown that that our cost
function is monotone and increasing along any given ray, in the following
proposition, we will show that it must be continuous.

\begin{proposition}\label{thm:continuous}
  A monotone, measurable scale invariant cost function on
  $\mathbb{R}^d$ can't have a discontinuity of size $k$ for any $k>
  0$.
\end{proposition}

\begin{proof}
  Assume that the cost function has a discontinuity of size $k$ at  $\eek$.
  Consider a configuration of two red points followed by two blue
  points with distances as shown in \cref{fig:pts4cont}.
    There are two possible minimal matchings here,
  the nested edges and the entwined edges
 (see \cref{fig:twopossible}), with costs $f(x_1)
  +f(x_4)$ and $f(x_2)+f(x_3)$ respectively.
  Now we will choose the four distances $x_1,x_2,x_3,x_4$ and the
  rescaling $s$ carefully around the discontinuity of the cost
  function to contradict scale invariance of the matching (See \cref{fig:cont}).
Since $f$ is monotone, it must have left limits and right limits
  at $\eek$. Therefore, there must exist  $\epsilon>0$ such that $f(\eek-\epsilon) >
  f(\eek^-)-k$ and $f(\eek+\epsilon)<f(\eek^+)+k$.

  \begin{figure}
    \centering
    \begin{subfigure}{0.9\textwidth}
      \centering
    \begin{tikzpicture}
   %   \filldraw [red] (0,0) circle (2pt);
   %   \filldraw [red] (1,0) circle (2pt);
            \draw[mark=x, color=red,thick,mark options={scale=1.5}]
            plot coordinates {(0,0)};
            \draw[mark=x, color=red,thick,mark
            options={scale=1.5}] plot coordinates {(1,0)};
      \filldraw [blue] (3,0) circle (2pt);
      \filldraw [blue] (4,0) circle (2pt);
      \draw[<->] (0,1/2) -- (3,1/2) node[midway,above]{$x_2$};
      \draw[<->] (1,1) -- (4,1) node[midway,above] {$x_3$};
      \draw[<->] (1,-1/2)--(3,-1/2) node[midway,below]{$x_1$};
      \draw[<->] (0,-1)--(4,-1) node[midway,below] {$x_4$};
    \end{tikzpicture}
  \caption{The configuration of points and their distances.} \label{fig:pts4cont}
\end{subfigure}

\begin{subfigure}{0.9\textwidth}
      \begin{tikzpicture}
       \draw [->] (0,0.1) [out=30,in=150] to (4,0.1);
       \draw[->] (1,0.1) [out=30,in=150] to (3,0.1);
            \draw [->] (6,0.1) [out=30,in=150] to (9,0.1);
      \draw[->] (7,0.1) [out=30,in=150] to (10,0.1);
     % \filldraw [red] (0,0) circle (2pt);
            \draw[mark=x, color=red,thick,mark options={scale=1.5}]
       plot coordinates {(0,0)};
       % \filldraw [red] (1,0) circle (2pt);
             \draw[mark=x, color=red,thick,mark options={scale=1.5}]
       plot coordinates {(1,0)};
      \filldraw [blue] (3,0) circle (2pt);
      \filldraw [blue] (4,0) circle (2pt);
   %   \filldraw [red] (6,0) circle (2pt);
            \draw[mark=x, color=red,thick,mark options={scale=1.5}]
       plot coordinates {(6,0)};
       % \filldraw [red] (7,0) circle (2pt);
             \draw[mark=x, color=red,thick,mark options={scale=1.5}]
       plot coordinates {(7,0)};
      \filldraw [blue] (9,0) circle (2pt);
      \filldraw [blue] (10,0) circle (2pt);
      \node [rectangle,fill=none,draw=none] at (2,-0.5) {nested
        edges};
       \node [rectangle,fill=none,draw=none] at (8,-0.5) {entwined
        edges};
    \end{tikzpicture}
    \caption{The two possible minimal matchings.}\label{fig:twopossible}
\end{subfigure}
\caption{The configuration of points we use to prove continuity
  of the cost function, along with the two possible minimal matchings.
  The nested edges have cost $f(x_1)+f(x_4)$ and the entwined edges
  have cost $f(x_2)+f(x_3)$.}
\end{figure}
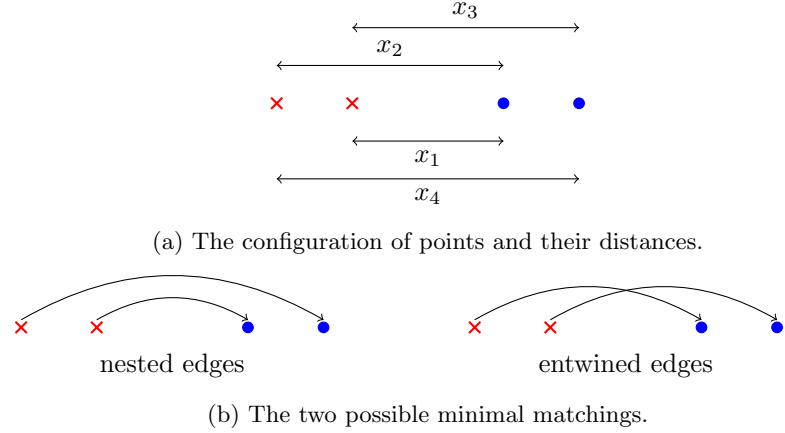

  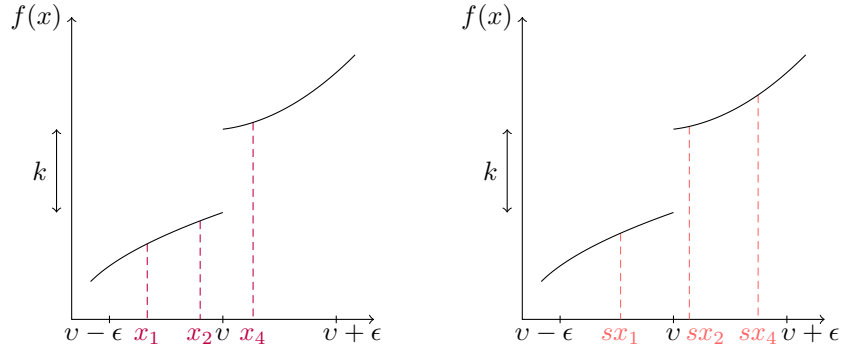
\begin{figure}
   % \centering
    \begin{tikzpicture}[domain=0:4]
      \draw[<->] (0,4)--(0,0)--(4,0);
      %\draw[<->] (5,4)--(5,0)--(9,0);
      \draw (2,-0.05)--(2,0.05) node[midway,below]{$\eek$};
      \draw (0.5,-0.05)--(0.5,0.05) node[anchor=north east,xshift=0.3cm]{$\eek-\epsilon$};
      \draw (3.5,-0.05)--(3.5,0.05) node[anchor=north west,xshift=-0.3cm]{$\eek+\epsilon$};
      %\draw (1,-0.05)--(1,0.05);
      %\draw (1.7,-0.05)--(1.7,0.05);
      %\ancdraw (2.8,-0.05)--(2.8,0.05);
      \draw[densely dashed,purple] (1,1)--(1,0) node
      [anchor=north]{$x_1$};
      \draw[densely dashed, purple] (1.7,1.3)--(1.7,0) node
      [anchor=north]{$x_2$};
      \draw[densely dashed, purple](2.4,2.605)--(2.4,0) node
      [left,anchor=north]{$x_4$};
      \draw[domain=0.25:2,smooth] plot ({\x},{sqrt(\x)});
      \draw[domain=2:3.75,smooth] plot ({\x},{0.25*(\x-1.75)^2+2.5});
      \draw[<->] (-0.2,1.414)--(-0.2,2.5156)
      node[midway,left]{$k$};
      \node[anchor=east] at (0,4){$f(x)$};
     % \node[anchor=west] at (5,0){$x$};
\end{tikzpicture}
\qquad
\begin{tikzpicture}
      \draw[<->] (0,4)--(0,0)--(4,0);
      %\draw[<->] (5,4)--(5,0)--(9,0);
      \draw (2,-0.05)--(2,0.05) node[midway,below]{$\eek$};
      \draw (0.5,-0.05)--(0.5,0.05) node[anchor=north east,xshift=0.3cm]{$\eek-\epsilon$};
      \draw (3.5,-0.05)--(3.5,0.05) node[anchor=north west,xshift=-0.2cm ]{$\eek+\epsilon$};
      %\draw (1,-0.05)--(1,0.05);
      %\draw (1.7,-0.05)--(1.7,0.05);
      %\ancdraw (2.8,-0.05)--(2.8,0.05);
      \draw[densely dashed,red!60] (1.3,1.14)--(1.3,0) node
      [anchor=north]{$sx_1$};
      \draw[densely dashed,red!60] (1.7*1.3,2.55)--(1.7*1.3,0) node
      [anchor=north west,xshift=-.17cm]{$sx_2$};
      \draw[densely dashed, red!60](2.4*1.3,2.969)--(2.4*1.3,0) node
      [left,anchor=north]{$sx_4$};
      \draw[domain=0.25:2,smooth] plot ({\x},{sqrt(\x)});
      \draw[domain=2:3.75,smooth] plot ({\x},{0.25*(\x-1.75)^2+2.5});
      \draw[<->] (-0.2,1.414)--(-0.2,2.5156)
      node[midway,left]{$k$};
      \node[anchor=east] at (0,4){$f(x)$};
     % \node[anchor=west] at (4.5,0){$x$};
\end{tikzpicture}
    \caption[test]{The cost function $f$ with a discontinuity at $\eek$.
    The diagram on the left displays the costs of the two matchings
    with the points unscaled. The diagram on the right shows the costs
    after rescaling by $s$. Before we rescale, the minimal matching
    will be given by the entwined edges since $f(x_1)+f(x_4)>2f(x_2)$ and after the rescaling it
    will be given by the nested edges since
    $f(sx_1)+f(sx_4)<2f(sx_2)$. This proves that a scale invariant
    cost function must be continuous along each ray.
    }
    \label{fig:cont}
  \end{figure}

First, we choose $x_1\in(\eek-\frac{\epsilon}{2},\eek)$ and $x_4\in \left( \eek,\left(
    \frac{\eek+\epsilon}{\eek}x_1 \right) \wedge (2\eek-x_1) \right)$. We can
    do this because $\frac{\eek+\epsilon}{\eek}x_1 >
    \frac{(\eek+\epsilon)\left(\eek-\tfrac{\epsilon}{2}\right)}{\eek} =
    \eek+\frac{\epsilon(\eek-\epsilon)}{2\eek} > \eek$ and $2\eek-x_1> \eek$.
    Now, we choose $x_2=x_3=\frac{x_1+x_4}{2} \in (x_1,\eek)$. This is
    true since $x_4< 2\eek-x_1$ ensures that $\frac{x_1+x_4}{2}
    = x_2=x_3 <\eek$.
    For this configuration, we have
    \begin{align*}
      f(x_2)+f(x_3) &< 2f(\eek^-) \qquad \mbox{ and }\\
        f(x_1)+f(x_4) &> f(\eek-\epsilon)+ f(\eek^-)+k > f(\eek^-)+k+f(\eek^-)-k =
                      2f(\eek^-),
    \end{align*}
    implying that the entwined edges form the minimal matching (see
    \cref{fig:cont}).

    However, if we consider the rescaling given by any $s\in\left(
    \frac{\eek}{x_2},\frac{\eek}{x_1} \right)$, we have $sx_1\in
  (\eek-\epsilon,\eek)$ and $sx_2,sx_3,sx_4\in (\eek,\eek+\epsilon)$ resulting
    in
    \begin{align*}
      f(sx_2)+f(sx_3)&> f(\eek^+)+f(\eek^+) = 2f(\eek^-)+2k\qquad \mbox{ and } \\
            f(sx_1)+f(sx_4) &< f(\eek^-)+f(\eek+\epsilon)<f(\eek^-)+f(\eek^+)+k =
                        2f(\eek^-)+2k,
    \end{align*}
    implying that the nested edges form the minimal matching (see
    \cref{fig:cont}), giving
    us a contradiction.
  \end{proof}

 So far, we have established that the cost function defined along any
ray given by $f:(0,\infty)\rightarrow\mathbb{R}$ will be monotone (\cref{thm:monotone}) and
continuous (\cref{thm:continuous}), implying that the cost function
will be locally $L^1$. Now, we define a multiplicative convolution as follows.

Choose $\phi:(0,\infty)\rightarrow (0,\infty)$ such that $\phi$ is
compactly supported and smooth (done explicitly in
\cref{approxid}). Define the multiplicative convolution of $f$ and
$\phi$ to be given by $f\star\phi:(0,\infty)\rightarrow
\mathbb{R}$ where
\begin{equation*}
  f\star \phi(x) = \int_0^\infty f(sx)\phi(s)ds.
\end{equation*}

\begin{claim}[{multiplicative }convolution]\label{convolution}
  If $f:(0,\infty)\rightarrow \mathbb{R}$ is a scale invariant cost
  function, and $\phi$ is as defined above, $f\star\phi$ must also be a
  scale invariant cost function.
\end{claim}

\begin{proof}
  We will show that for $f$ scale invariant, the function $g(x) = f(x) +
  f(kx) $ will also be scale invariant for any $k>0$. A similar argument can be used
  to show that any positive linear combination of $f(k_i x)$s  will be
  scale invariant and then defining the integral as a Riemann
  integral, we can conclude that $f\star\phi$ must be scale invariant.

  Fix $k>0$.
  Consider a collection of red points $X=\{x_1,\ldots,x_n\}$, a
  collection of blue points $Y=\{y_1,\ldots, y_n\}$ and a minimal
  matching $M$ induced by
  the cost function $f(x)$. This means that for any other matching $N$, we have
  \begin{align}
    \label{eq:temp1}
    \sum_{i=1}^n f(M(x_i)-x_i )\leq  \sum_{i=1}^n f(N(x_i)-x_i).
  \end{align}
  Now, scale invariance of the minimal matching implies
  \begin{equation}
    \label{eq:temp25}
    \sum_{i=1}^n f(kM(x_i)-kx_i) \leq  \sum_{i=1}^n f(k N(x_i)-kx_i).
  \end{equation}

  If we consider the cost function $g(x) = f(x)+f(kx)$, adding
  \cref{eq:temp1,eq:temp25}, we can conclude that the $g$-minimal
  matching will also
  be $M$ since for any other matching $N$,
  \begin{align*}
    \sum_{i=1}^n f(M(x_i)-x_i) + f(k M(x_i)-kx_i) & \leq
    \sum_{i=1}^n f(N(x_i)-x_i) +f(kN(x_i) -kx_i)\nonumber\\
    \sum_{i=1}^n g(M(x_i)-x_i) &\leq  \sum_{i=1}^n g(N(x_i)-x_i).
  \end{align*}
  Finally we can rescale equations \cref{eq:temp1} and \cref{eq:temp25}
  by $k>0$ due to the scale invariance of $f$, add them and conclude
  that $g(x)$ is also scale invariant.
\end{proof}

Now we want to show that $f\star\phi$ for $\phi\in
C^\infty_c(0,\infty)$ is smooth. We can do this by performing a change
of variables ($\log s=b$)
to write it as a convolution in the usual sense, which we know to be
smooth. Let
\begin{align*}
  %\label{eq:conv}
  g(x) & = \int_0^\infty f(sx) \phi(s)ds\nonumber \\
       & =  \int_0^\infty f(exp(\log s+\log x)) \phi(s)ds\nonumber \\
\Rightarrow \quad g(e^y) & =   \int_0^\infty f(exp(\log s+y)) \phi(s)ds\nonumber \\
  & = \int_{-\infty}^\infty f(exp(b + y))\phi(e^b) e^b db. \nonumber
\end{align*}
Since $ \phi(e^b) e^b$ is smooth and compactly supported and
$f(exp(y+b))$ is locally $L^1$, we conclude that $g=f\star\phi$ is smooth.
Now, we use the following result adapted
from  \cite[Proposition 18]{minimal} to classify what a smooth,
monotone, non-constant, non-decreasing cost function could look like.

\begin{proposition}\label{scinvfrommin}
  Let $g:(0,\infty)\rightarrow\R$ be a scale invariant cost function that is smooth,
  non-constant and non-decreasing. Then,
   there exist
$a,b,\gamma\in\mathbb{R}$, such that
\begin{equation}\label{eq:cost1dim}
  g(x) =
  \begin{cases}
    ax^\gamma +b, & \gamma\neq 0, \qquad \mbox{or} \\
    a\log x + b, & \gamma=0.
  \end{cases}
\end{equation}
\end{proposition}
The above proposition differs from \cite[Proposition 18]{minimal}
because it discusses cost functions along the positive $x$ axis instead of a symmetric
cost function on $\R$. We prove this modified result at the end of
this section.
\cref{scinvfrommin} tells us that any function
$f\star \phi$, where $f$ is a scale invariant cost function and $\phi$
is smooth, must be of the form in \eqref{eq:cost1dim}.
In order to show that $f$ must also be one of the above two
possibilities in \eqref{eq:cost1dim}, we choosing an
approximate identity as follows.

\begin{definition}[approximate identity]\label{approxid}
  An approximate identity is a family of functions
  $\{\phi_\epsilon\}_{0<\epsilon<1}\in C^\infty_c(0,\infty)$ such that
  \begin{enumerate}[(i)]
  \item $\int_0^\infty \phi_\epsilon(a) da =1$ for all $\epsilon>0$,
  \item $\phi_\epsilon(x)\geq 0$ for all $x,\epsilon>0$,
    \item $\mbox{supp}(\phi_\epsilon) = (1-\epsilon,1+\epsilon)$.
  \end{enumerate}
\end{definition}
Now, it is a simple exercise in real analysis to show that $f\star\phi_\e$
converges pointwise to $f$ as $\e$ approaches $0$. We leave this as an exercise to the reader.
\iffalse
\begin{claim}[pointwise convergence]\label{ptwiseconvergence}
  For an approximate identity from \cref{approxid} and
  $f:(0,\infty)\rightarrow \R$, any continuous function, we have $f\star\phi_\epsilon\rightarrow f$
  pointwise.
\end{claim}

\begin{proof}
  Fix $x>0$ and $\epsilon_0>0$. Since $f$ is continuous at $x$, there
  exists $\delta>0$ such that for $|x-y|<\delta$, we have
  $|f(x)-f(y)|<\epsilon_0$.

  If we choose $\epsilon<\epsilon(x)=\delta/x$, we have

  \begin{align*}
    %\label{eq:ptw}
    |f\star\phi_\epsilon(x)-f(x)|& = \left|\int f(ax)\phi_\epsilon(a)da -
                                   f(x)\right|\nonumber \\
                                 &= \left|\int \left(f(ax)-f(x)\right) \phi_\epsilon(a)da\right|\nonumber\\
    &\leq \int_{|a-1|<\epsilon}  |f(ax)-f(x)| |\phi_\epsilon(a)|da
      \nonumber\\
    &\leq \int_{|a-1|<\epsilon} \epsilon_0\phi_\epsilon(a)da \leq
    \epsilon_0.
  \end{align*}

  The last inequality here holds because if $|a-1|<\epsilon<\delta/x$,
  we have $|ax-x|< \delta$ implying $|f(ax)-f(x)|<\epsilon_0$.

  Since $\epsilon_0>0$ was chosen arbitrarily, we have
  $f\star\phi_\epsilon$ converges to $f$ pointwise for all $x>0$ or
  equivalently
  \[
\forall x>0, \qquad \lim_{\epsilon\rightarrow 0} f\star\phi_\epsilon(x) =
f(x).
  \]
\end{proof}
\fi
This in addition to \cref{scinvfrommin}, tells us that $f$ is the
pointwise limit of a sequence of functions with the form described in
\cref{eq:cost1dim}. In
\cref{recoverf} below we show that $f$ must also have this form.

\begin{proposition}[Recovering $f$ from the convolution]\label{recoverf}
  Assume $f:(0,\infty)\rightarrow\mathbb{R}$ is a measurable scale
  invariant cost function along a ray. Then, there exist
  $a,b,\gamma\in\mathbb{R}$ such that
  \[
f(x) =
\begin{cases}
  ax^\gamma +b,& \gamma\neq 0, \qquad \mbox{or} \\
  a\log x +b,& \gamma =0.
\end{cases}
\]
\end{proposition}

\begin{proof}
  We established earlier that $f$ must be monotone
  (\cref{thm:monotone}), continuous (\cref{thm:continuous}), and
  the pointwise limit of the convolution above
  (\cref{approxid}).
  We know from \cref{scinvfrommin} that for each $\epsilon>0$, there exist
  $a_\e,b_\e,\gamma_\e\in\mathbb{R}$ such that
  \[
\fpe(x) =
\begin{cases}
  a_\e
  x^{\gamma_\e} +b_\e,& \gamma_e\neq 0, \qquad \mbox{or} \\
  a_\e\log x +b_\e, & \gamma_\e =0,
\end{cases}
\]
and $f\star\phi_\epsilon(x)\rightarrow f$ pointwise. We consider the
following two cases.

{\bf Case 1:} There are infinitely many $\e>0$ , such that
$\fpe(x) =a_\e \log x + b_\e$.
For simplicity, we index this sequence by $\e$. Due to pointwise
convergence of this sequence $\fpe$ to $f$ for all $x>0$, we have
\begin{align}
  f(1)& = \lim_{\e\rightarrow 0} \fpe(1) =  \lim_{\e\rightarrow 0} a_\e
  \log 1 +b_\e =  \lim_{\e\rightarrow 0} b_\epsilon, \label{eq:f1} \\
  f(e)& = \lim_{\e\rightarrow 0} \fpe(e) =  \lim_{\e\rightarrow 0} a_\e
  \log e +b_\e =  \lim_{\e\rightarrow 0}a_\e + b_\epsilon. \label{eq:fe}
\end{align}
 Now, \cref{eq:f1} implies $ \lim_{\e\rightarrow 0} b_\epsilon $ exists and equals
 $f(1)$. This, in addition to \cref{eq:fe} implies
 $\lim_{\e\rightarrow 0} a_\e $ also exists and equals $f(e)-f(1)$.
 Since the subsequence $\fpe(x) = \aep \log x +\be$ converges to
 $f(x)$ pointwise and both $\aep$ and $\be$ converge, we conclude that
 for $a = f(e)-f(1)$ and $b = f(1)$, we have
 \[f(x) = a\log x +b. \]

 {\bf Case 2:} There are infinitely many $\epsilon>0$ such that
 $\fpe=a_\e x^{\gamma_e} + b_\e$. A similar argument to the one above
 can be used to show that $\lim_{\e\rightarrow 0}a_\e$,
 $\lim_{\e\rightarrow 0}b_\e$ and $\lim_{\e\rightarrow 0}\gamma_\e$ all exist and that $f$ must equal
 $a x^\gamma +b$ where $a$, $b$ and $\gamma$ are the respective limits.
\end{proof}

This concludes that any measurable scale invariant cost function along
a given ray must
be of the form as in \cref{eq:cost1dim}.
We are finally ready to finish the proof of \cref{thm:unbalcost}, which discusses how different the cost function can be along different
rays.

\begin{proof}[Proof of \cref{thm:unbalcost}]
  Consider a configuration of three points in $\mathbb{R}^n$, with a
  red point and two blue points such that the two possible edges are
  along different rays as in \cref{fig:unbalcostlast}.
  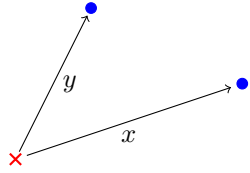
\begin{figure}[h]
    \centering

  \begin{tikzpicture}
    %\filldraw [red] (0,0) circle (2pt);
          \draw[mark=x, color=red,thick,mark options={scale=1.5}] plot
          coordinates {(0,0)};

    \filldraw [blue] (3,1) circle (2pt);
    \filldraw [blue] (1,2) circle (2pt);
    \draw[->] (0+.15,0.05) -- (3-0.15,1-0.05) node[midway,below] {$x$};
    \draw[->] (0+0.05,0+0.1) -- (1-0.05,2-0.1) node[midway,right] {$y$};
  \end{tikzpicture}
\caption{A configuration of points used in the proof of
  \cref{thm:unbalcost}.} \label{fig:unbalcostlast}
  \end{figure}

  Let the cost functions along the two rays be
  given by $f_{\ti}$ and $f_{\tii}$ respectively.
From \cref{recoverf}, we know that $f_{\ti} $ and $f_{\tii}$ must each
have the form of  \eqref{eq:cost1dim}.
  We choose the distances
  between the points to be given by $x,y\in(0,\infty)$
  such that the two edges have the same cost. This, along with the
  scale invariance of the cost function gives us the following two
  equations,
  \begin{align}
    \label{eq:sc2dir}
    f_{\ti}(x) &= f_{\tii}(y),\\
     f_{\theta_1}(kx) &= f_{\theta_2}(ky). \label{eq:sc2dir2}
  \end{align}
  In this proof, we vary $k>0$ and show that for $\gamma\neq 0$, the parameters $\gamma$
  and $b$ in \eqref{eq:cost1dim}
  must be equal in all directions, and in case the cost
  function resembles a  logarithm, the parameter $a$ of \eqref{eq:cost1dim} should be the
  same in all directions.

  {\bf Case 1:} If $f_{\theta_1}(x) = a x^{\gamma} +b$ and $f_{\theta_2}(y) =
  A y^{\Gamma} +B$.\\
  From \cref{eq:sc2dir,eq:sc2dir2}, we
  know
  \begin{align}
    a  x^{\gamma} +b &= A  y^{\Gamma} +B, \nonumber\\
    a k^{\gamma} x^{\gamma} +b &= A k^{\Gamma} y^{\Gamma} +B, \quad
    \mbox{for all } k>0. \nonumber
  \end{align}
  Comparing the LHS and RHS of the second equation as
  $k\rightarrow\infty$, we get $\Gamma=\gamma$. Then,
  taking the limit as $k\rightarrow 0$ for $\gamma>0$ or the limit as
  $k\rightarrow\infty$ for $\gamma<0$, we get $b=B$. So the only parameter which can be different along
  different rays is the coefficient $a$.

  {\bf Case 2:} If $f_{\theta_1}(x) = a \log(x) +b$ and
  $f_{\theta_2}(y) = A\log(y)+B$.\\
  From \cref{eq:sc2dir,eq:sc2dir2}, we
  know
  \begin{align}
    a\log(x) +b &= A\log(y) +B, \nonumber\\
    a(\log(k)+ \log(x)) +b &= A(\log(k)+\log(y)) +B, \qquad \mbox{for
    all } k>0.\nonumber
 \end{align}
  Taking the difference, we get
  $A\log(k)=a\log(k)$ for all $k>0$ implying that $A=a$.

  {\bf Case 3:} If $f_\ti(x) =ax^\gamma+b$ and $f_\tii(y) =
  A\log(y)+B$.\\
  From \cref{eq:sc2dir,eq:sc2dir2}, we
  know
  \begin{align}\label{eq:cros}
    ax^\gamma +b &= A\log y +B\nonumber\\
    ak^\gamma x^\gamma +b &= A\log k+A\log y +B, \qquad \mbox{for all
    }k>0
  \end{align}
Note that the LHS of \cref{eq:cros} is $O(k^\gamma)$ and the RHS of
\cref{eq:cros} is $O(\log k)$, these cannot be equal unless we had
$A=0$ which results in $f_\ti$ and $f_\tii$ both being the same
constant function.
\end{proof}

Now we present the proof of \cref{scinvfrommin} which we split into
several lemmas. This is very similar to the proof in \cite[Proposition
18]{minimal}, but done carefully so that we may avoid the cost
function being symmetric. We denote the function $f$ in the following way
  \begin{equation}\label{eq:phantom}
    f(x) =
    \begin{cases}
      g_2(x) & \text{ for } x < 0,\\
      g_1(x) & \text{ for } x\geq 0,
    \end{cases}
  \end{equation}
  for $g_1:[0,\infty)\rightarrow \mathbb{R}$ increasing and
  $g_2:(-\infty,0)\rightarrow \mathbb{R}$ decreasing.
{
 Note that the cost function cannot be constant on an interval
 without being constant along the entire ray. This can be seen easily
 using a configuration of one red point and two blue points with
 distances that can be rescaled to be in the interval where the cost
 function is constant and then outside it where there is a clear
 minimal matching.
}

  \begin{claim} \label{cl:ranges}
    The range of $g_1$ and the range of $g_2$ must have some
    intersection.
  \end{claim}

\begin{proof}
    Suppose not. Say the cost function is continuous along rays
    (\cref{thm:continuous}) and the two branches have disjoint ranges, WLOG
    i.e.,
    $\inf\{g_2\}<\sup\{g_2\}<\inf\{g_1\}<\sup\{g_1\}$, and
    $\sup\{g_2\},\inf\{g_1\}\in\R$. This means that the shortest
    possible red-blue edge is more expensive than the longest possible
    blue-red edge. Suppose the minimal matching contains a red-blue
    edge $\langle r,b \rangle$ with $r<b$. We consider the following
    cases, pictured in \cref{fig:ranges2},
    regarding any other red point at $r'$ where $r'>r$.

    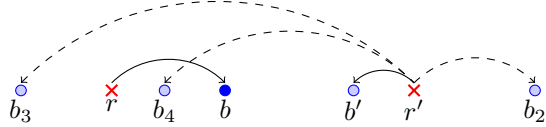
\begin{figure}[h]
      \centering
      \begin{tikzpicture}
        \foreach \i in {0,4}{
         % \filldraw[red] (\i,0) circle (2pt);
          \draw[mark=x, color=red,thick,mark options={scale=1.5}] plot coordinates {(\i,0)};
          }
          \filldraw[blue] (1.5,0) circle (2pt);
          \foreach \i in {-1.2,0.7,3.2,5.6}
          {\filldraw[blue,fill opacity = 0.2] (\i,0) circle (2pt);}
          \draw[->] (0,0.1) to [out = 45, in=135] (1.5,0.1);
          \draw[<-] (3.2,0.1) to [out = 45,in=135] (4,0.1);
          \foreach \i in {-1.2,0.7}
          {\draw[<-,dashed] (\i,0.1) to [out=45,in=135] (4,0.1);}
          \draw[->,dashed] (4,0.1) to [out=45,in=135] (5.6,0.1);
          \node[below] at (0,0) {$r$};
          \node[below] at (1.5,0) {$b$};
          \node[below] at (4,0) {$r'$};
          \node[below] at (3.2,0){$b'$};
          \node[below] at (5.6,0){$b_2$};
          \node[below] at (-1.2,0){$b_3$};
          \node[below] at (0.7,0){$b_4$};
      \end{tikzpicture}
      \caption{In \cref{cl:ranges}, we show that a minimal matching
        for a cost function with disjoint ranges of $g_1$ and $g_2$
        cannot have a single red-blue edge in general. If there was a
        red-blue edge $\langle r,b \rangle$, then any red point to the
        right of this edge $r'$ could only be matched to a blue point
        $b'$ such that $r<b<b'<r'$. This can only occur for specific
        point configurations, and occurs with probability zero for
        Poisson point processes.   }
      \label{fig:ranges2}
    \end{figure}

    First, the red point $r'$ cannot be unmatched since
    matching it to the blue point $b$ would reduce the cost of the
    matching.
    Second, the point $r'$ cannot be matched to a blue point at $b_2$
    such that  $r<b<r'<b_2$ since
    if we consider the scaling where $k\rightarrow 0$, the cost of
    this matching $\{\langle kr,kb \rangle,\langle kr',kb_2 \rangle\}$ approaches $2\inf\{g_1\}$ and the cost of the matching $\{\langle kr,kb_2 \rangle,\langle kb,kr' \rangle\}$ approaches
    $\inf\{g_1\}+\inf\{g_2\}$ which is smaller given our assumption.
    Third, if the red point at $r'$ was matched to a blue point at
    $b_3$ where $b_3<r<b<r'$, then on taking the rescaling so
    $k\rightarrow\infty$, the cost of the matching $\{\langle kr,kb
    \rangle,\langle kb_3,kr' \rangle\}$ is given by $2\sup\{g_2\}$
    which is smaller than the cost of the matching $\{\langle kr,kb_3
    \rangle,\langle kr',kb \rangle\}$ which approaches
    $\sup\{g_2\}+\sup\{g_1\}$.
    Finally, if the red point $r'$ is matched to a blue point at $b_4$
    such that $r<b_4<b<r'$, simply by monotonicity of the cost function
    (\cref{thm:monotone}), we see that the matching $\{\langle r',b
    \rangle,\langle r,b_4 \rangle\}$ is minimal.
    Therefore, if there is a red-blue edge $\langle r,b \rangle$, then any red point to the
    right of the edge $r'$ can only be matched to a blue point $b'$
    such that $r<b<b'<r'$. This imposes a specific condition on the
    points where the number of red points can never exceed the number
    of blue points in the interval $(b,x)$ for any $x\in(0,\infty)$.
    This is not satisfied if we consider matchings between two Poisson
    point processes of equal intensity on $\R$.

    We could still have matchings which do not contain any red-blue
    edges. These are studied in \cref{thm:a-infty}.
  \end{proof}

  Since the range of
  $g_1$ and the range of $g_2$ have some intersection,
  there exists $t_1>0$ and $t_2<0$ such that $f(t_1)=f(t_2)$. We can replace $f(x)$ with
  $\frac{f(x)}{f(t_1)}$ so that we have  $f(t_1)=f(t_2)=1$.
  Since $g_2$ is decreasing and $0>t_2>2t_2-t_1$, we have $g_1(t_1) =
  g_2(t_2) < g_2(2t_2-t_1)$. Consequently, there exists $\delta_2>0$
  such that $2g_2(t_2-\delta_2) < g_2(t_2) + g_1(2t_1-t_2)$.
  Similarly, because $g_1$ is increasing and $t_1<2t_1-t_2$, we
  have $g_2(t_2)=g_1(t_1)< g_1(2t_1-t_2)$. So, there exists
  $\delta_1>0$ such that $2g_1(t_1+\delta_1)< g_1(t_1) +
  g_2(2t_2-t_1)$.
  Fix $\delta=\min{\{\delta_1,\delta_2}\}$ for the rest of this argument.

  \begin{lemma}[Analogous to Lemma 19 in \cite{minimal}] \label{lem:sc-inv}
    For $x,y\in[t_1,t_1+\delta]$ and $w,z\in[t_2-\delta,t_2]$, and any
    $s>0$, $f(x)+f(y)\leq f(w)+f(z)$ implies
    \begin{equation}
      f(sx)+f(sy)\leq f(sw)+f(sz).
    \end{equation}
  \end{lemma}

  \begin{proof}
    We construct a line with alternating blue and red points on it
    with distances $x,w,y,z$ between them respectively as shown in
    Figure \ref{img:lem.sc-inv}.
  It suffices to show that under the hypothesis of the Lemma, the
    matching whose cost is given by $f(x)+f(y)$ is the minimal
    matching. This in addition to scale invariance proves the lemma.

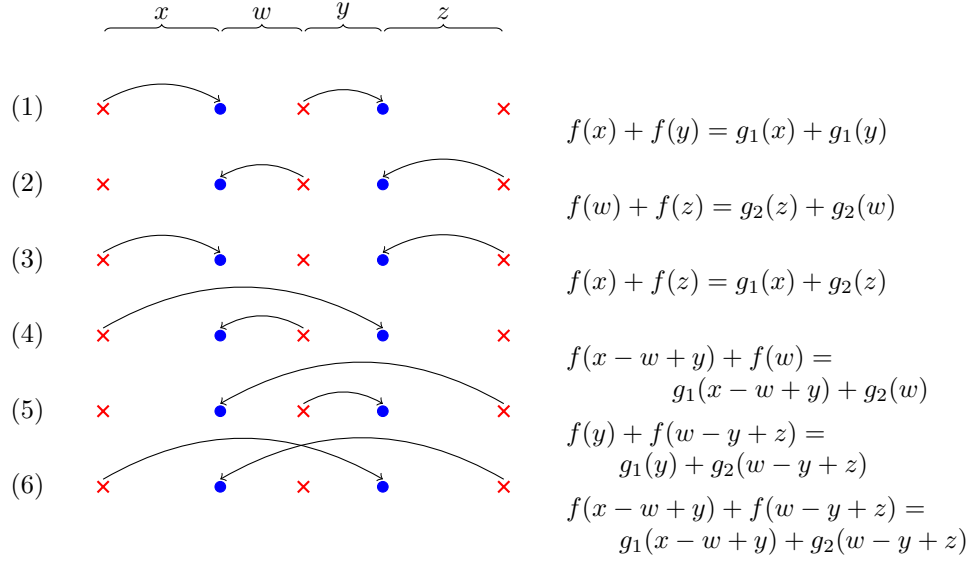
\begin{figure}
    \begin{tikzpicture}%\label{img:lem.sc-inv}
      \foreach \i in {1,...,6}
      {
      %  \draw (-1.5/2,-\i) -- (12.1/2,-\i);
      %\draw [->,green] (0,0) to [out = 315,in=225] (5.5,0);
%      %  \filldraw[red] (0,-\i) circle (2pt);
        \draw[mark=x, color=red,thick,mark options={scale=1.5}] plot coordinates {(0,-\i)};
        \filldraw[blue] (3.1/2,-\i) circle (2pt);
       % \filldraw[red] (5.3/2,-\i) circle (2pt);
          \draw[mark=x, color=red,thick,mark options={scale=1.5}] plot
      coordinates {(5.3/2,-\i)};
        \filldraw[blue] (7.4/2,-\i) circle (2pt);
  %      \filldraw[red] (10.6/2,-\i) circle (2pt);
          \draw[mark=x, color=red,thick,mark options={scale=1.5}] plot
      coordinates {(10.6/2,-\i)};
      }
      \foreach \j in {1,...,6}{\node at (-1,-\j) {($\j$)};}
      \draw[decoration={brace,raise=5pt},decorate]
      (0.02,-0.15) -- node[above=6pt] {$x$} (3.1/2-0.02,-0.15);
      \draw[decoration={brace,raise=5pt},decorate]
      (3.1/2+0.02,-0.15) -- node[above=6pt] {$w$} (5.3/2-.02,-0.15);
      \draw[decoration={brace,raise=5pt},decorate]
      (5.3/2+0.02,-0.15) -- node[above=6pt] {$y$} (7.4/2-0.02,-0.15);
      \draw[decoration={brace,raise=5pt},decorate]
  (7.4/2+0.02,-0.15) -- node[above=6pt] {$z$} (10.6/2-0.02,-0.15);
     % \draw [<->] (0,0.15) -- (3.1/2,0.15) node[midway,above]{$x$};
     % \draw [<->] (3.1/2,0.15) -- (5.3/2,0.15) node[midway,above]{$w$};
     % \draw [<->] (5.3/2,0.15) -- (7.4/2,0.15) node[midway,above]{$y$};
     % \draw [<->] (7.4/2,0.15) -- (10.6/2,0.15) node[midway,above]{$z$};
      \draw [->] (0,-1+0.1) [out = 30, in = 150] to (3.1/2,-1+0.1);
      \draw [->] (5.3/2,-1+0.1) [out=30,in=150] to (7.4/2,-1+0.1);
      \draw [<-] (3.1/2,-2+0.1) [out=30,in=150] to (5.3/2,-2+0.1);
      \draw [<-] (7.4/2,-2+0.1) [out=30,in=150] to (10.6/2,-2+0.1);
      \draw [->] (0,-3+0.1) [out=30,in=150] to (3.1/2,-3+0.1);
      \draw [<-] (7.4/2,-3+0.1) [out=30,in=150] to (10.6/2,-3+0.1);
      \draw [->] (0,-4+0.1) [out=30,in=150] to (7.4/2,-4+0.1);
      \draw [<-] (3.1/2,-4+0.1) [out=30,in=150] to (5.3/2,-4+0.1);
      \draw [<-] (3.1/2,-5+0.1) [out=30,in=150] to (10.6/2,-5+0.1);
      \draw [->] (5.3/2,-5+0.1) [out=30,in=150] to (7.4/2,-5+0.1);
      \draw [<-] (3.1/2,-6+0.1) [out=30,in=150] to (10.6/2,-6+0.1);
      \draw [->] (0,-6+0.1) [out = 30, in =150] to (7.4/2,-6+0.1);
      \node[anchor=north west] at (6,-1) [align=left] {$f(x)+f(y)=g_1(x)+g_1(y)$};
      \node[anchor=north west] at (6,-2) [align=left] {$f(w)+f(z)=g_2(z)+g_2(w)$};
      \node[anchor=north west] at (6,-3) [align=left] {$f(x)+f(z)= g_1(x)+g_2(z)$};
      \node[anchor=north west] at (6,-4) [align=left] {$f(x-w+y)+f(w) =$\\\qquad\qquad$ g_1(x-w+y)+g_2(w)$ };
      \node[anchor=north west] at (6,-5) [align=left] {$f(y)+f(w-y+z) =$\\\qquad$
        g_1(y)+g_2(w-y+z)$};
      \node[anchor= north west] at (6,-6) [align=left] {$f(x-w+y)+f(w-y+z) = $\\\qquad$
      g_1(x-w+y)+g_2(w-y+z)$};
    \end{tikzpicture}
    \caption{Possible matchings and costs for the configuration
    of the points used in the proof of
    \cref{lem:sc-inv}. } \label{img:lem.sc-inv}
    \end{figure}

   From the hypothesis we know that the cost of Matching (1) is
   at most the cost of Matching (2).
   Without loss of generality, assume $x\geq y$ and $-z\geq -w$ this implies $f(x)\geq f(y)$ and
   $f(z)\geq f(w)$. Now if we have $f(y)>f(z)$, this leads to  a contradiction
   since we must have $f(x)+f(y)\leq f(z)+f(w)$. So we must have
   $f(y)\leq f(z)$ which implies $f(y)+f(x)\leq f(z)+f(x)$ which means
   the cost of Matching (1) is at most the cost of Matching (3),
   implying Matching (3) cannot be the minimal matching. In addition
   to this, since Matching (6) has edges with costs strictly larger
   than those in Matching (3), it cannot be minimal either.

   Now, the maximum value of $f(z)+f(w)$ is
   $2g_2(t_2-\delta)$ since $z$ and $w$ were chosen to be in
   $[t_2-\delta,t_2]$ and $f$ is decreasing in this interval. And the
   minimum value of $f(x-w+y)+f(w)$ is $g_2(t_2)+g_1(2t_1-t_2)$ by a
   similar argument. Due to our choice of $\delta$,
   $2g_2(t_2-\delta)\leq g_2(t_2)+g_1(2t_1-t_2)$ which implies that
   the cost of Matching (2) is at most the cost of Matching (4),
   implying Matching (4) cannot be the minimal matching.

   Finally, the maximum value of $f(x)+f(y)$ is $2g_1(t_1+\delta)$ and
   the minimum value of $f(y)+f(w-y+z)$ is $g_1(t_1)+g_2(2t_2-t_1)$.
   Due to our choice of $\delta$ we have that $2g_1(t_1+\delta)\leq
   g_1(t_1)+g_2(2t_2-t_1)$ which implies that the cost of Matching (1)
   is at most the cost of Matching (5), implying that Matching (5)
   cannot be the minimal matching.

   Therefore, we know that for these points and distances, Matching (1) must be the minimal
   matching.
  \end{proof}

  Note that the proof above only works when $x,y\geq 0$ and $z,w\leq
  0$.

  \begin{lemma}[Analogous to Lemma 20 in \cite{minimal}]\label{lem:t1delta}
    For $z,x\in[t_1,t_1+\delta]$ and $s>0$ \[\frac{g_1'(sx)}{g_1'(x)} =
      \frac{g_1'(sz)}{g_1'(z)} ,\] where
    \[
      f(x) =
      \begin{cases}
        g_1(x), & x\geq 0, \\
        g_2(x), & x<0.
      \end{cases}
    \]
\end{lemma}

\begin{proof}
Let $\delta$ be as defined earlier.
Now, we choose $\epsilon$ small enough so that $f(x)+\epsilon <
f(t_1+\delta)$ and $f(z)+\epsilon < f(t_1+\delta)$.
Let $w=h_1(f(x)+\epsilon)$ and $y=h_1(f(z)+\epsilon)$.
Now choose $u,v\in[t_2-\delta,t_2]$ such that $u=v=h_2 \left(
  \frac{f(x)+f(y)}{2} \right)$. Now we have
$f(z)+f(w) = f(x)+\epsilon+f(z)=f(x)+f(y) = f(u)+f(v)$.

We can use Lemma \ref{lem:sc-inv} twice  to get $f(su)+f(sv) =
f(sx)+f(sy)$ and then twice again to get $f(su)+f(sv) = f(sz)+f(sw)$.
Note that to use Lemma \ref{lem:sc-inv}, we need four distances, two
positive and two negative.
Therefore we have,
\begin{align*}
  f(sx)+ f(sy) &= f(sz)+f(sw)\\
  g_1(sx)+g_1(sh_1(g_1(z) +\epsilon)) & =
                                        g_1(sz)+g_1(sh_1(g_1(x)+\epsilon)) \\
  \frac{g_1(sh_1(g_1(z) +\epsilon)) - g_1(sz)}{\epsilon} &
  =\frac{g_1(sh_1(g_1(x)+\epsilon)) - g_1(sx) }{\epsilon}.
\end{align*}
Now since this is true for all $\epsilon$ small enough, we take the
limit as $\epsilon\rightarrow 0$ to get
\begin{align*}
  (g_1 \circ sh_1)'(g_1(z)) & =  (g_1 \circ sh_1)'(g_1(x)) \\
  g_1'(sh_1(g_1(z))) sh_1'(g_1(z)) & = g_1'(sh_1(g_1(x)))
                                     sh_1'(g_1(x)) \\
  g_1'(sz) s h_1'(g_1(z)) & = g_1'(sx) s h_1'(g_1(x))\\
  \frac{g_1'(sz)}{g_1'(z)} &= \frac{g_1'(sx)}{g_1'(x)}.
\end{align*}
\end{proof}

Now, we can follow the proofs almost identically as in
\cite[Lemma 21, Lemma 22, Proposition 18]{minimal}, to
conclude that $g_1$ must have the same form as in \eqref{eq:cost1dim}.

\section{Comparison with Balanced matchings (Proof of
  \cref{compwbal})} \label{sec:compwbal}

In this section, we
introduce several lemmas, some of which have been proven in
\cite{minimal} for balanced matchings, and then prove \cref{compwbal}.
Assume $R$ and $B$ are drawn from Poisson point processes of intensity
$1$ on $\R$.

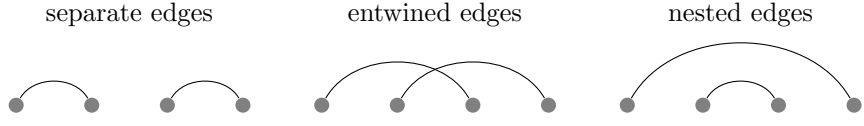
\begin{figure}
  \begin{tabular}{ccccc}
\begin{tikzpicture}
  \node[circle,inner sep=2pt, minimum size = 2pt,fill=gray] (a) at (0,0){};
  \node[circle, inner sep=2pt, minimum size = 2pt, fill=gray] (b) at (1,0){};
  \node[circle , inner sep=2pt, minimum size = 2pt ,fill=gray] (c) at (2,0){};
  \node[circle, inner sep=2pt, minimum size = 2pt , fill=gray] (d) at (3,0){};
  \draw (a) to [bend left=60] (b);
  \draw (c) to [bend left=60] (d);
  \node[rectangle,fill=none,draw=none] at (1.5,1.2){separate edges};
\end{tikzpicture}& \hfill &
    \begin{tikzpicture}
  \node[circle, inner sep=2pt, minimum size = 2pt ,fill=gray] (a) at
  (0,0){};

  \node[circle,inner sep=2pt, minimum size = 2pt,fill=gray] (b) at (1,0){};
  \node[circle,inner sep=2pt, minimum size = 2pt,fill=gray] (c) at (2,0){};
  \node[circle,inner sep=2pt, minimum size = 2pt, fill=gray] (d) at (3,0){};
  \draw (a) to [bend left=60] (c);
  \draw (b) to [bend left=60] (d);
  \node[rectangle,fill=none,draw=none] at (1.5,1.2){entwined edges};
\end{tikzpicture}
    & \hfill &
      \begin{tikzpicture}
  \node[circle,inner sep=2pt, minimum size = 2pt,fill=gray] (a) at (0,0){};
  \node[circle,inner sep=2pt, minimum size = 2pt,fill=gray] (b) at (1,0){};
  \node[circle,inner sep=2pt, minimum size = 2pt,fill=gray] (c) at (2,0){};
  \node[circle,inner sep=2pt, minimum size = 2pt, fill=gray] (d) at (3,0){};
  \draw (a) to [bend left=60] (d);
  \draw (b) to [bend left=60] (c);
  \node[rectangle,fill=none,draw=none] at (1.5,1.2){nested edges};
\end{tikzpicture}\\
    \end{tabular}
\caption{Any two edges in a matching can be classified as either separate, entwined or nested edges}\label{fig:edges}
\end{figure}

\newcommand{\f}[1]
{\begin{tikzpicture}[scale=.75,thick]
  \tikzstyle{every node}=[circle,
                        inner sep=0pt, minimum width=2mm]
  \useasboundingbox (-.5,-1) rectangle (3.5,1);
  #1
  \end{tikzpicture} }
\newcommand{\ff}[1]
{\begin{tikzpicture}[scale=.75,thick]
    \tikzstyle{every node}=[circle,
    inner sep=0pt, minimum width=2mm]
  \useasboundingbox (-.5,-1.3) rectangle (3.5,1);
  #1
  \end{tikzpicture} }
\newcommand{\rrrr}{
  \draw[mark=x, color=red,thick] plot coordinates {(0,0)};
  \draw[mark=x, color=red,thick] plot coordinates {(.8,0)};
  \draw[mark=x, color=red,thick] plot coordinates {(1.8,0)};
    \draw[mark=x, color=red,thick] plot coordinates {(3,0)};
}
\newcommand{\rbbr}{
  \draw[mark=x, color=red,thick, mark options={scale=1.7}] plot coordinates {(0,0)};
  \node[](a) at (0,0){};
\node[fill=blue,minimum size = 0.15cm](b) at (.8,0){};
\node[fill=blue,minimum size = 0.15cm](c) at (1.8,0){};
\draw[mark=x, color=red,thick ,mark options={scale=1.7}] plot coordinates {(3,0)};
\node[](d) at (3,0){};
}
\newcommand{\brbr}{
\node[fill=blue,minimum size = 0.15cm](a) at (0,0){};
\node[](b) at (.8,0){};
\draw[mark=x, color=red,thick ,mark options={scale=1.7}] plot
coordinates {(0.8,0)};
\node[fill=blue,minimum size = 0.15cm](c) at (1.8,0){};
\node[](d) at (3,0){};
\draw[mark=x, color=red,thick ,mark options={scale=1.7}] plot coordinates {(3,0)};
}
\newcommand{\bbrr}{
\node[fill=blue,minimum size = 0.15cm](a) at (0,0){};
\node[fill=blue,minimum size = 0.15cm](b) at (.8,0){};
\node[](c) at (1.8,0){};
\draw[mark=x, color=red,thick ,mark options={scale=1.7}] plot
coordinates {(1.8,0)};
\node[](d) at (3,0){};
\draw[mark=x, color=red,thick ,mark options={scale=1.7}] plot coordinates {(3,0)};
}
\newcommand{\brrb}{
\node[fill=blue,minimum size = 0.15cm](a) at (0,0){};
\node[](b) at (.8,0){};
\draw[mark=x, color=red,thick ,mark options={scale=1.7}] plot
coordinates {(0.8,0)};
\node[](c) at (1.8,0){};
\draw[mark=x, color=red,thick ,mark options={scale=1.7}] plot
coordinates {(1.8,0)};
\node[fill=blue,minimum size = 0.15cm](d) at (3,0){};
}
\newcommand{\rbrb}{
  \node[](a) at (0,0){};
  \draw[mark=x, color=red,thick ,mark options={scale=1.7}] plot coordinates {(0,0)};
\node[fill=blue,minimum size = 0.15cm](b) at (.8,0){};
\node[](c) at (1.8,0){};
\draw[mark=x, color=red,thick ,mark options={scale=1.7}] plot
coordinates {(1.8,0)};
\node[fill=blue ,minimum size = 0.15cm](d) at (3,0){};
}
\newcommand{\rrbb}{
  \node[](a) at (0,0){};
  \draw[mark=x, color=red,thick ,mark options={scale=1.7}] plot
  coordinates {(0,0)};
  \draw[mark=x, color=red,thick ,mark options={scale=1.7}] plot
  coordinates {(0.8,0)};
\node[](b) at (.8,0){};
\node[fill=blue,minimum size = 0.15cm](c) at (1.8,0){};
\node[fill=blue,minimum size = 0.15cm](d) at (3,0){};
}
\newcommand{\sep}{
\draw (a) to [bend left=60] (b);
\draw (c) to [bend left=60] (d);
}
\newcommand{\ent}{
\draw (a) to [bend left=60] (c);
\draw (b) to [bend left=60] (d);
}
\newcommand{\str}{
\draw[densely dashed] (b) to [bend right=60] (c);
\draw[densely dashed] (a) to [bend right=60] (d);
}
\newcommand{\sstr}{
\draw (b) to [bend right=60] (c);
\draw (a) to [bend right=60] (d);
}
\newcommand{\rightab}{\draw[->] (a) to [bend left=60](b);}
\newcommand{\rightac}{\draw[->] (a) to [bend left = 60](c);}
\newcommand{\rightad}{\draw[->] (a) to [bend right=60](d);}
\newcommand{\rightbc}{\draw[->] (b) to [bend right=60](c);}
\newcommand{\rightbd}{\draw[->] (b) to [bend left=60](d);}
\newcommand{\rightcd}{\draw[->] (c) to [bend left=60](d);}

\newcommand{\leftab}{\draw[<-] (a) to [bend left=60](b);}
\newcommand{\leftac}{\draw[<-] (a) to [bend left=60](c);}
\newcommand{\leftad}{\draw[<-] (a) to [bend right=60](d);}
\newcommand{\leftbc}{\draw[<-] (b) to [bend right=60](c);}
\newcommand{\leftbd}{\draw[<-] (b) to [bend left=60](d);}
\newcommand{\leftcd}{\draw[<-] (c) to [bend left=60](d);}

\newcommand{\dleftad}{\draw[densely dashed, <-] (a) to [bend
  right=60](d);}
\newcommand{\dleftbc}{\draw[densely dashed, <-] (b) to [bend right =
  60](c);}
\newcommand{\drightad}{\draw[densely dashed,->] (a) to [bend right =
  60](d);}
\newcommand{\drightbc}{\draw[densely dashed,->] (b) to [bend
  right=60](c);}

\begin{figure}
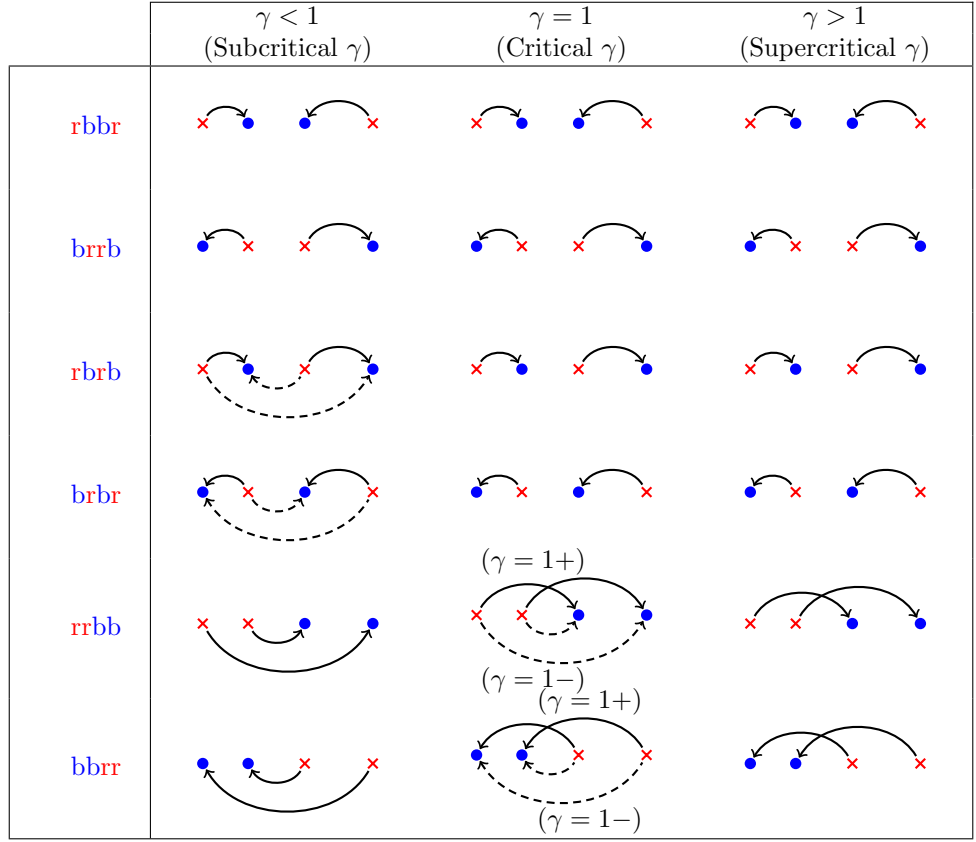
\label{fig:1dim}
\centering
\begin{tabular}{|>{\centering\arraybackslash} m{.01cm}
>{\centering\arraybackslash} m{1cm}|
>{\centering\arraybackslash} m{3.2cm}>{\centering\arraybackslash} m{3.2cm}>{\centering\arraybackslash} m{3.2cm}|}
  \cline{3-5}

  \multicolumn{2}{c|}{}& $\gamma<1$ & $\gamma=1$ & $\gamma>1$ \\
  \multicolumn{2}{c|}{}& (Subcritical $\gamma$) & (Critical $\gamma$) & (Supercritical $\gamma$) \\

  \hline
 & \textcolor{red}{r}\textcolor{blue}{bb}\textcolor{red}{r} & \f{\rbbr\rightab\leftcd} & \f{\rbbr\rightab\leftcd} & \f{\rbbr\rightab\leftcd} \\
   & \textcolor{blue}{b}\textcolor{red}{rr}\textcolor{blue}{b} & \f{\brrb\leftab\rightcd} & \f{\brrb\leftab\rightcd} & \f{\brrb\leftab\rightcd} \\
                        & \textcolor{red}{r}\textcolor{blue}{b}\textcolor{red}{r}\textcolor{blue}{b} &
   \f{\rbrb\rightab\rightcd\drightad\dleftbc} & \f{\rbrb\rightab\rightcd} & \f{\rbrb\rightab\rightcd} \\

 & \textcolor{blue}{b}\textcolor{red}{r}\textcolor{blue}{b}\textcolor{red}{r} &
 \f{\brbr\leftab\leftcd\dleftad\drightbc} & \f{\brbr\leftab\leftcd} & \f{\brbr\leftab\leftcd} \\

 & \textcolor{red}{rr}\textcolor{blue}{bb} & \f{\rrbb\rightad\rightbc} &
                                                              \ff{\rrbb\rightac\rightbd\drightad\drightbc
    \node[rectangle,fill=none,draw=none] at (1.0,.95){$(\gamma=1+)$};
    \node[rectangle,fill=none,draw=none] at (1.0,-1.15){$(\gamma=1-)$};
 }
 &\f{\rrbb\rightac\rightbd} \\

 & \textcolor{blue}{bb}\textcolor{red}{rr} & \f{\bbrr\leftad\leftbc} &
 \ff{\bbrr\leftac\leftbd\dleftad\dleftbc
 \node[rectangle,fill=none,draw=none] at (2,.95){$(\gamma=1+)$};
 \node[rectangle,fill=none,draw=none] at (2,-1.15){$(\gamma=1-)$};
 }
 & \f{\bbrr\leftac\leftbd} \\
 \hline
\end{tabular}
  \caption{Possible arrangements (separate, nested, or entwined) of
  two edges in a $\gamma$-minimal matching on $\R$.  The figures
  indicate only the order of the points, not their distances.  Solid
  lines and dashed lines indicate two different possible matchings;
  in the bottom row with $\gamma=1$ both possibilities are $1$-minimal
  (with the tie broken as indicated for $\gamma=1\pm$), while in the
  other two cases the minimal choice depends on the distances between
  points, and the values of the two parameters $a$ and $\gamma$.} \label{possibilities}
\end{figure}

Any two edges of a matching on $\R$ must be either
entwined, nested, or separate, as shown in \cref{fig:edges}. In
\cref{possibilities}, we list possibilities for the minimal matching
for the different configurations of points.

For supercritical values
of $\gamma$ (i.e., $\gamma\in\{1^+\}\cup(1,\infty]$), given the
relative positions of the points,  there is only
one possibility for the minimal matching. This matching does not depend on
$\gamma$, $a$ or the distances between the points, it only depends on
the relative positions of the points. On the other hand, for subcritical $\gamma$ (i.e.,
$\gamma\in[-\infty,1)\cup\{1^-\}$) it is not as simple and there are configurations where the
minimal matching depends on the parameters and the distances between
the points. In \cref{super} below, we discuss what supercritical
minimal matchings look like.

\begin{lemma}\label{super}
  If the set of red
points is given by $R=\{ \cdots<r_{-1}<r_0<r_1<\cdots \}$ and the
set of blue points is given by $B=\{
\cdots<b_{-1}<b_0<b_1<\cdots\}$, any supercritical minimal matchings will be of the
form
\begin{equation}\label{eq:super}
\widetilde{M}_k:=\left\{ \langle r_{i+k},b_i \rangle :i\in\Z
\right\}
\end{equation}
for some $k\in\Z$.
\end{lemma}

\begin{proof}
  This result follows verbatim from the proof presented in
 \cite[Section 4]{minimal} for balanced supercritical
  matchings. The
  proof uses observations from \cref{possibilities} to conclude that
  for supercritical $\gamma$, the minimal matching must always
  preserve relative orders between the red and blue points (see
  \cref{fig:super} for examples of such matchings). These matchings
  are all non-invariant, non-factor matchings and locally finite.
\end{proof}

\begin{figure}[h]
  \centering
  \includegraphics[trim = {3cm 1cm 3cm
    2cm},clip,width=10cm]{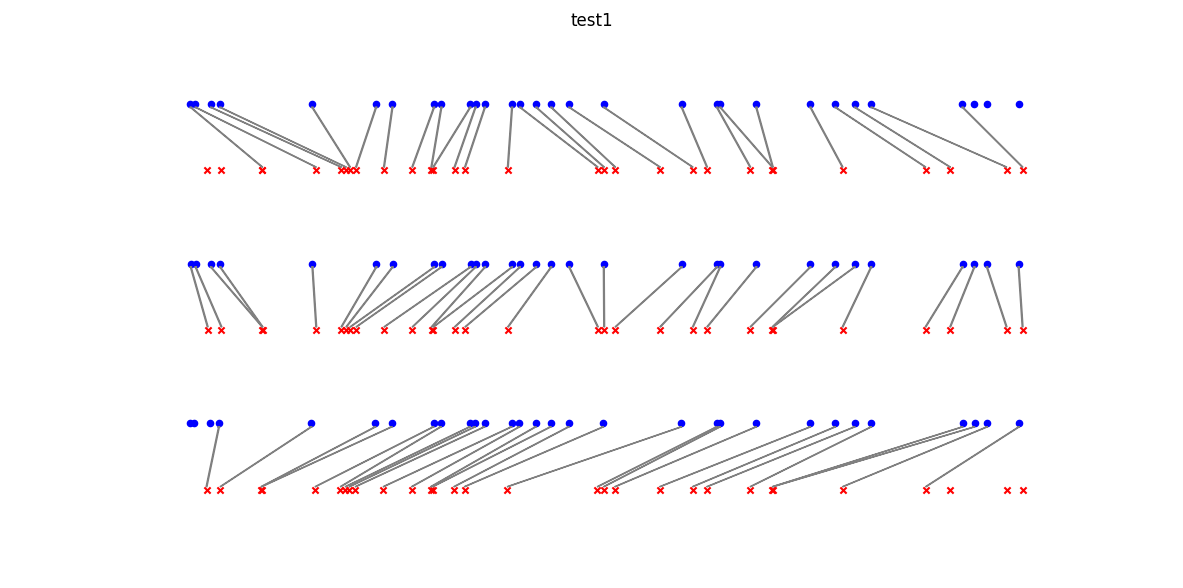}
  \caption{Possible $(\gamma,a)$-minimal matchings for the supercritical
    case, i.e., $\gamma\in(1,\infty)\cup\{1^+\}$. Note that the
    matchings all preserve the relative
   orders of the red and blue points. }\label{fig:super}
\end{figure}

Below, we introduce an interpretation of the minimal matchings which arise
from random walks. This extremely useful formulation is from
\cite{minimal}.
Given disjoint discrete
sets $R,B\subset\R$ we define the associated {walk} $W=W_{R,B}:\R\to\Z$ by
\begin{equation}
\begin{aligned}
  W(0-)&=-1 \qquad \mbox{and}\\
  W(y)-W(x)&=\#\bigl(R\cap(x,y]\bigr)-\#\bigl(B\cap(x,y]\bigr),\quad x<y.
\end{aligned}\label{walk}
\end{equation}
The walk takes a step up at every red point and down at every blue point. We
assume $W(0^-)=-1$ and $W(0^+)=0$ for convenience, since we use the
Palm process assumed to have a red point at $0$. If $R$
and $B$ are independent Poisson processes of intensity $1$ on $\R$, $W$
gives us a
continuous-time simple symmetric random walk. A possible draw of the
continuous time random walk has been shown in \cref{fig:rw}. This
partitions the points of the two Poisson point processes into levels,
which we define below.

\begin{figure}[h]
  \centering
  \includegraphics[trim={5cm 0 5cm 1.5cm},clip,width=12cm]{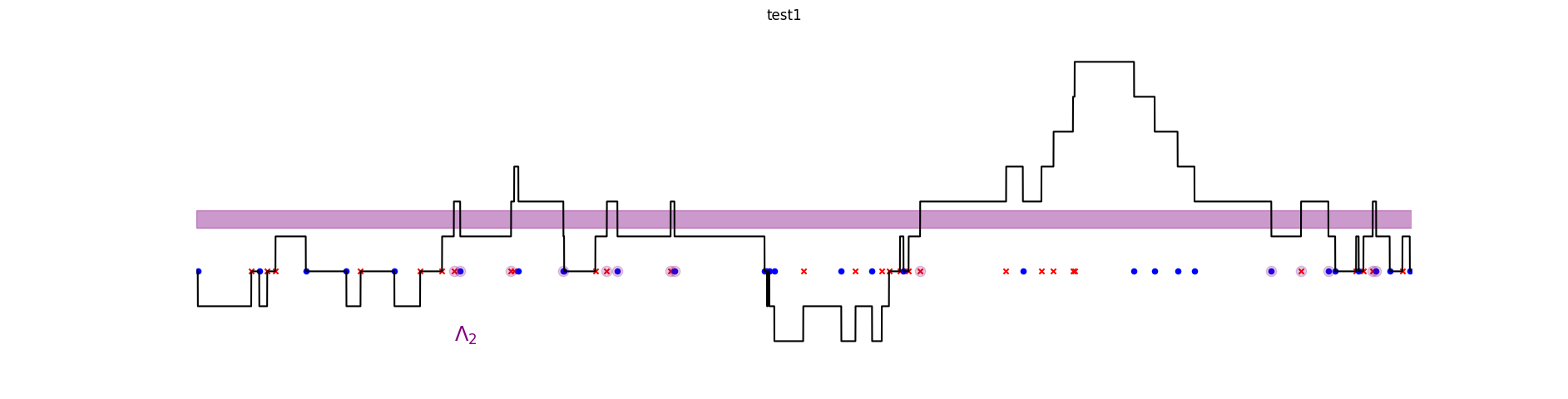}
  \caption{Random walk formulation for the matchings for the sub
    critical case $\gamma<1$. Highlighted is the level $\Lambda_2$
    along with all the points in it. For a subcritical minimal
    matching, points in each level are always matched to each other.}
  \label{fig:rw}
\end{figure}

For $k\in\Z$, we define {level} $k$ to be the set
$\Lambda_k=\Lambda_k(R,B)$ of points where the walk moves between $k$
and $k+1$
\begin{equation}\label{level}
  \Lambda_k:=\Bigl\{x\in \R:\{W(x-),W(x+)\}=\{k,k+1\}\Bigr\}.
\end{equation}
The levels $(\Lambda_k)_{k\in \Z}$ form a partition of
$R\cup B$ and within a given level, the elements alternate
between red and blue points. In \cref{fig:rw}, the level $\Lambda_2$
has been highlighted along with the points in that level.

\begin{lemma}[Lemma 25 of \cite{minimal}]\label{lem:levels}
  For subcritical minimal matchings, any given point will only be
  matched to other  points in the same level.
\end{lemma}

\begin{proof}
 From \cref{possibilities}, we know subcritical minimal matchings do
  not prefer entwined edges, and matching a point to another level
  introduces entwined edges. (See \cite[Section 4]{minimal} for a
 more detailed proof.)
\end{proof}

For $\gamma=1^-$, we know from \cref{possibilities} that separate
edges are always preferred over nested edges. Therefore, in the minimal
matching for a given level $\Lambda_j = \{\cdots<r_{-1}<b_{-1}<r_{0}<b_0<r_1<b_1<\cdots\}$, either each red point is matched to the blue point to its right
resulting in a matching entirely of red-blue edges or each red point
is matched to the blue point to its left resulting in a matching
 entirely of red-blue edges.  We call these possible matchings of the
 level $m^j_{rb} = \{\langle r_i,b_i \rangle\}_{i\in\Z}$ and  $
 m^j_{br} = \{\langle r_i,b_{i-1} \rangle\}_{i\in\Z}$ respectively.
 Both of these matchings
will be minimal regardless of the value of $a$ since the cost of the matching cannot be
reduced by a reordering of finitely many edges. Now, we discuss the
family of
$(1^-,a)$-minimal matchings.

\begin{lemma}\label{lem:1minus}
  For $\gamma=1^-$ and any $a\in(0,\infty)$, the family of
  $(1^-,a)$-minimal matchings will be given by
  $\{M_\infty,M_{-\infty}\}\cup\{M_k\}_{k\in\Z}$ where
  \begin{equation}
    \label{eq:1minus}
    M_{-\infty}  = \bigcup_{j\in\Z} m^j_{rb} , \quad
    M_{\infty} = \bigcup_{j\in\Z} m^j_{br} ,  \quad  \mbox{and} \quad M_k  =\left\{\bigcup_{j>k}
  m^j_{rb} \right\} \cup \left\{\bigcup_{j\leq k} m^j_{br}\right\}.
  \end{equation}
Here, $M_\infty$ and $M_{-\infty}$ are invariant and locally infinite
whereas all the $M_{k}$ are non-invariant and locally finite.
\end{lemma}

The proof of \cref{lem:1minus}, can be found in \cite[Section 5, Theorem
2(iii)]{minimal}. Since this proof only depends on the
random walk defined in \eqref{walk}, which only depends on the
configuration of the points, the same proof holds for unbalanced costs
as well.  Some possible
$(1^-,a)$-minimal matchings have been illustrated in \cref{fig:1minus}.

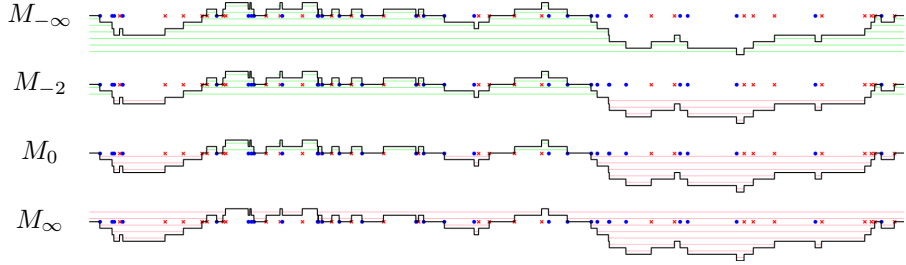
\begin{figure}[h]
  \centering
  \begin{tikzpicture}
    \node at (0,0) {\includegraphics[trim={6cm 2cm 5cm
    2cm},clip,width=11cm]{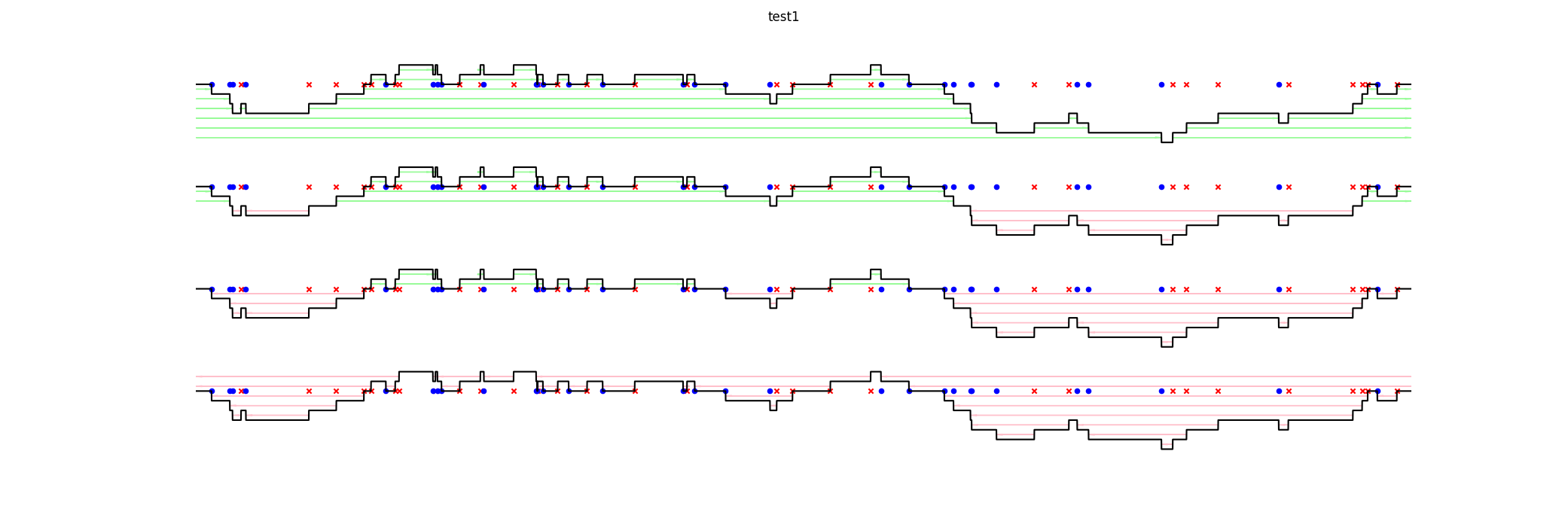}};
\draw (-6,1.5) node {$M_{-\infty}$};
\draw (-6,0.6) node {$M_{-2}$};
\draw (-6,-0.3) node {$M_0$};
\draw (-6,-1.2) node {$M_{\infty}$};
  \end{tikzpicture}
    \caption{Four possible $(1^-,a)$-minimal matchings. The first matching
  called $M_{\infty}$ consists only of red-blue edges, and the last
  matching called $M_{-\infty}$ consists only of blue-red edges. These
  two are locally infinite and invariant. The middle two matchings are
  $M_{-2}$ and $M_0$ respectively. They belong to the family of
  locally finite, non-invariant
  matchings $\{M_k\}_{k\in\Z}$, all of which are $(1^-,a)$-minimal
  matchings and consist of red-blue for all levels
  $\{\Lambda_n\}_{n\geq k}$ above level $k$ and blue-red edges for all levels
  $\{\Lambda_n\}_{n<k}$ below level $k$ (See \cref{lem:1minus}). }
  \label{fig:1minus}
\end{figure}

For $\gamma<1$, while \cref{lem:levels} holds, it is not necessarily
true that each level must only consist of separate edges. And indeed
there are sometimes nested edges in a given level.
Below we introduce an extremely useful property of $(\gamma,a)$-minimal matchings for $\gamma<1$ called
{\em quasistability}   which is
analogous to the notion of stability (introduced in \cite{galeshapley} and studied in the context of
matchings between Poisson points in \cite{hpps}) but with an extra multiplicative
constant $\kappa$. \cref{quasi} is the unbalanced analogue of
\cite[Proposition 3.2]{minimal}.

\begin{proposition}[Quasistability]\label{quasi}
  For each $\gamma\in[-\infty,1)$ and fixed $a\in(0,\infty)$, there exists
  $\kappa=\kappa(\gamma,a)\in[1,\infty)$ with the following property. If $M$ is a $(\gamma,a)$-minimal  matching between two sets
   $R,B\subset\R$
  and $x$ and $y$ are two distinct points of different colours not
  matched to each other then,
  \begin{equation}\label{kappa-ineq}
    |x-M(x)|\wedge|y-M(y)|\leq \kappa\,|x-y|.
  \end{equation}

\end{proposition}

The above proposition tells us while these subcritical matchings
are not quite stable (where \cref{kappa-ineq} would have to hold for
$\kappa=1$), the quasistability property ensures that for two points relatively
close by that are not matched to each other, at least one of them isn't
matched extremely far away.

\begin{proof}[Proof of \cref{quasi}]
 The points $x$ and $y$ cannot be both unmatched in a minimal matching
 since they are of different colors.  Moreover, if only $y$ (say) is unmatched then, since $\kappa\geq 1$, if \cref{kappa-ineq} fails then $|x-y|<|x-M(x)|$, so matching $x$ to $y$ instead of $M(x)$ would reduce the cost.  Therefore we  assume $x$ and $y$ are both matched.
  We can also assume that they are not matched to each other,
  otherwise \cref{kappa-ineq} holds trivially.
 Suppose that \cref{kappa-ineq} fails, i.e.\
 \begin{equation}|x-M(x)|,\;|y-M(y)|> \kappa|x-y|.\label{fails}\end{equation}
 We claim that, with the appropriate choice of $\kappa$, modifying $M$
 by matching instead the pairs $\langle x,y\rangle$ and $\langle
 M(x),M(y)\rangle$ strictly reduces the cost, in contradiction to
 minimality.
 We consider the following four cases.

\begin{leftbar}
{\bf Case 1:} For fixed $a$ and $\gamma=-\infty$, we can take
$\kappa=1$ and the claim is immediate since the minimal matching is stable.\\
{\bf Case 2: } For $-\infty<\gamma<0$, we choose $\kappa$ in the
following way. Define $\underline{a} = \inf \{a(\theta):\theta\in S^{d-1}\}$ and $\overline{a} =
\sup\{a(\theta):\theta\in S^{d-1}\}$ to be  the infimum and supremum values
of $a$ on the unit sphere. Now, we take any $\kappa>
(\frac{\underline{a}}{2\overline{a}})^{1/\gamma}$. If we assume that
the coefficients for edges $\langle x,M(x)
\rangle, \langle y,M(y) \rangle,\langle x,y \rangle, \langle M(x),M(y)
\rangle$ are $a_1,a_2,a_3,a_4$ respectively, comparing the costs of
the two matchings we get
 \begin{align*}
   -a_1|x-M(x)|^\gamma-a_2|y-M(y)|^\gamma&\leq- \overline{a} 2 \kappa^\gamma |x-y|^\gamma\\
   <-\underline{a} |x-y|^\gamma&\leq- a_3|x-y|^\gamma-a_4 |M(x)-M(y)|^\gamma.
 \end{align*}
This implies the matching with edges $\langle x,y \rangle$ and
$\langle M(x),M(y) \rangle$ is minimal, contradicting our assumption
that $M$ is minimal.
\end{leftbar}
 For the remaining cases $0\leq \gamma<1$, we normalize so that the distance
 between $x$ and $y$ is $1$, and the length of the edge $\langle
 x,M(x) \rangle$ is $u$ and the length of the edge $\langle y,M(y)
 \rangle$ is $v$. The assumption in \cref{fails} becomes $u,v>\kappa$ and
 due to the triangle inequality we have $$  |M(x)-M(y)|\leq
 |x-y|+|x-M(x)|+|y-M(y)| = 1+u+v.$$

   \begin{figure}[h]
     \centering
     \begin{tikzpicture}[scale=\textwidth/20cm]
       \node[circle,fill=black,inner sep =-1.5pt,label=left:$x$] (x1)
       at (11,1){};
       \node[circle,draw,fill=white,inner sep = -1.5pt,label=right:$y$] (y1)
       at (12,1){};
       \node[circle,fill=white,draw,inner sep = -1.5pt, label=left:$M(x)$]
       (Mx1) at (16,1){};
       \node[circle,fill=black,inner sep = -1.5pt, label=right:$M(y)$]
       (My1) at (18,1){};
       \draw[->] (x1) [out = 45,in=135] to (y1);
       \draw[<-] (Mx1) [out=45,in=135] to (My1);
       \draw[->,dashed] (x1) [out=-45,in=-135] to (Mx1);
       \draw[<-,dashed] (y1) [out=-45,in=-135] to (My1);
% second line of points
       \node[circle,fill=black,inner sep =-1.5pt,label=left:$x$] (x2)
       at (-1,1){};
       \node[circle,draw,fill=white,inner sep = -1.5pt,label=right:$y$] (y2)
       at (0,1){};
       \node[circle,fill=white,draw,inner sep = -1.5pt, label=right:$M(x)$]
       (Mx2) at (6,1){};
       \node[circle,fill=black,inner sep = -1.5pt, label=left:$M(y)$]
       (My2) at (4,1){};
       \draw[->] (x2) [out = 45,in=135] to (y2);
       \draw[<-] (Mx2) [out=180-45,in=45] to (My2);
       \draw[->,dashed] (x2) [out=-45,in=-135] to (Mx2);
       \draw[<-,dashed] (y2) [out=-45,in=-135] to (My2);
       % third line of points
        \node[circle,fill=black,inner sep =-1.5pt,label=left:$x$] (x3)
       at (1.5,3){};
       \node[circle,draw,fill=white,inner sep = -1.5pt,label=right:$y$] (y3)
       at (2.5,3){};
       \node[circle,fill=white,draw,inner sep = -1.5pt, label=left:$M(x)$]
       (Mx3) at (-1,3){};
       \node[circle,fill=black,inner sep = -1.5pt, label=right:$M(y)$]
       (My3) at (6,3){};
       \draw[->] (x3) [out = 45,in=135] to (y3);
       \draw[<-] (Mx3) [out=30,in=150] to (My3);
       \draw[->,dashed] (x3) [out=180+45,in=-45] to (Mx3);
       \draw[<-,dashed] (y3) [out=-45,in=-135] to (My3);
       %fourth row of points
            \node[circle,fill=black,inner sep =-1.5pt,label=left:$x$] (x3)
       at (13.5,3){};
       \node[circle,draw,fill=white,inner sep = -1.5pt,label=right:$y$] (y3)
       at (14.5,3){};
       \node[circle,fill=black,inner sep = -1.5pt, label=left:$M(y)$]
       (My3) at (11,3){};
       \node[circle,fill=white,draw,inner sep = -1.5pt, label=right:$M(x)$]
       (Mx3) at (18,3){};
       \draw[->] (x3) [out = 45,in=135] to (y3);
       \draw[<-] (Mx3) [out=180-30,in=30] to (My3);
       \draw[->,dashed] (x3) [out=-45,in=-135] to (Mx3);
       \draw[<-,dashed] (y3) [out=180+45,in=-45] to (My3);
       \node at (8.75,1){(iv)};
       \node at (8.75,3){(ii)};
       \node at (-3.75,1){(iii)};
       \node at (-3.75,3){(i)};
    \end{tikzpicture}
     \caption{ In order
     to prove quasistability (\cref{quasi}) for matchings on $\R$ with
     $\gamma\in(0,1)$, we consider the above four configurations. We
     show that  if the edges $\langle x,M(x) \rangle$ and $\langle y,M(y)
     \rangle$ get arbitrarily large, the minimal matching will switch to
     $\{\langle x,y \rangle, \langle M(x),M(y) \rangle\}$.}\label{fig:quasipf}
   \end{figure}
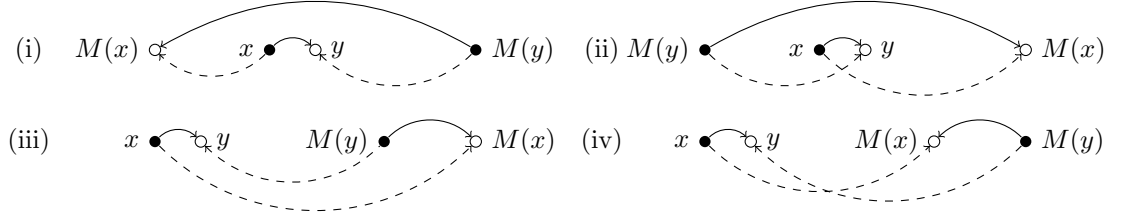

\begin{leftbar}
 {\bf Case 3:} In the case $\gamma=0$, the cost function is of the
 form $\log(x)+b(\theta)$. We assume that the different values of $b$
 along the different rays have
 $\underline{b}=\inf\{b(\theta):\theta\in S^{d-1}\}$ and
 $\overline{b}=\sup\{b(\theta):\theta\in S^{d-1}\}$. Then, the change in cost associated
 with the modification is
 \begin{align*}
   \log|x-y|&+b_3+\log|M(x)-M(y)|+b_4-\log|x-M(x)|-b_1-\log|y-M(y)|-b_2
   \\
  %  & \leq  \log 1+b_3+\log|1+u+v|+b_4-\log|u|-b_1-\log|v|-b_2\\
   & \leq \log\left( \frac{|x-y||M(x)-M(y)|}{|x-M(x)||y-M(y)|} \right)
     + 2(\overline{b}-\underline{b})\\
            & \leq \log\left(\frac{1+u+v}{uv} \right) +
   2(\overline{b}-\underline{b}) \\
   &=\log\Bigl(\frac{1}{uv}+\frac{1}{u}+\frac{1}{v}\Bigr)
   +2(\overline{b}-\underline{b}).
 \end{align*}
 We can take $\kappa>\frac{3}{e^{-5(\overline{b}-\underline{b})}}$, so
 that the above expression $\log\Bigl(\frac{1}{uv}+\frac{1}{u}+\frac{1}{v}\Bigr)
   +2(\overline{b}-\underline{b})<0$.\\
   {\bf Case 4:}
   Finally for $0<\gamma<1$ for $d=1$, we consider the four point
   configurations on $\R$ as illustrated in \cref{fig:quasipf}. We
 assume that $|x-y|=1,|x-M(x)|=u$ and $|y-M(y)|=v$. For the
 configuration in
 \cref{fig:quasipf}(i), the difference of costs of the two matchings
 is given by $$-a_1 u^\gamma-a_2v^\gamma+a_1+a_2(v+1-u)^\gamma,$$ which
 is negative for all $u>1$. For points as in \cref{fig:quasipf}(ii),
 all the edges have the same coefficient and since $\gamma<1$, the
 minimal matching avoids the entwined edges. For
 \cref{fig:quasipf}(iii), the change in cost is given by
 \begin{align}
   \label{eq:temp0}
   &  a_1+a_1(u-v-1)^\gamma - a_1u^\gamma-a_2 v^\gamma\nonumber \\
   &= (a_1-a_2v^\gamma) +a_1((u-v-1)^\gamma-u^\gamma),\nonumber
 \end{align}
which is negative for $v$ large enough.
Finally in \cref{fig:quasipf}(iv), the change in cost is given by $$
a_1+ a_2(v+1-u)-a_1u^\gamma-a_2v^\gamma,$$ which is negative for any
$u>1$.

\end{leftbar}
\end{proof}

\begin{remark}[Quasistability in higher dimensions]
  The proof of \cref{quasi} above, works in higher dimensions for all
  values of $\gamma$ outside of $(0,1)$. For $\gamma\in(0,1)$, if we
  consider the difference in cost of the two matchings and use
  triangle inequality, we can show that the change in cost is at most
  $\overline{a}\left[ 1+(1+u+v)^\gamma \right]-\underline{a}\left[
  u^\gamma+v^\gamma \right]$, which for $u=v\rightarrow \infty$
  approaches $\underline{a}u^\gamma\left[
    \frac{\overline{a}}{\underline{a}} 2^\gamma-2 \right]$.
If $\log_2\left(
  \frac{\overline{a}}{\underline{a}} \right)<1-\gamma$, the above
  expression will be negative which implies quasistability holds.
 % expression to be negative, we require $\log_2\left(
  %  \frac{\overline{a}}{\underline{a}} \right)<1-\gamma$.
  Therefore, in
  higher dimensions, the only $(\gamma,a)$-minimal matchings where quasistability might not
  hold are those with $\gamma\in(0,1)$ and $\log\left(
  \frac{\overline{a}}{\underline{a}} \right)>1-\gamma$.
\end{remark}

\begin{proposition}[Proposition 30 of \cite{minimal}]\label{prop:eta}
  Let $R$ and $B$ be independent Poisson Point processes of intensity
  $1$ on $\R$, define the random walk $W$ as in \cref{walk}.
  Fix $\eta = 2\kappa+1 >1$. There exists an almost surely positive, finite
  random variable $Y=Y_\eta$ such that $W>0$ on $[-\eta
  Y,-Y]\cup[Y,\eta Y]$.
\end{proposition}

\begin{figure}[h]
  \centering
  \includegraphics[trim={6cm 1.5cm 6cm 1.5cm},clip,width=\textwidth]{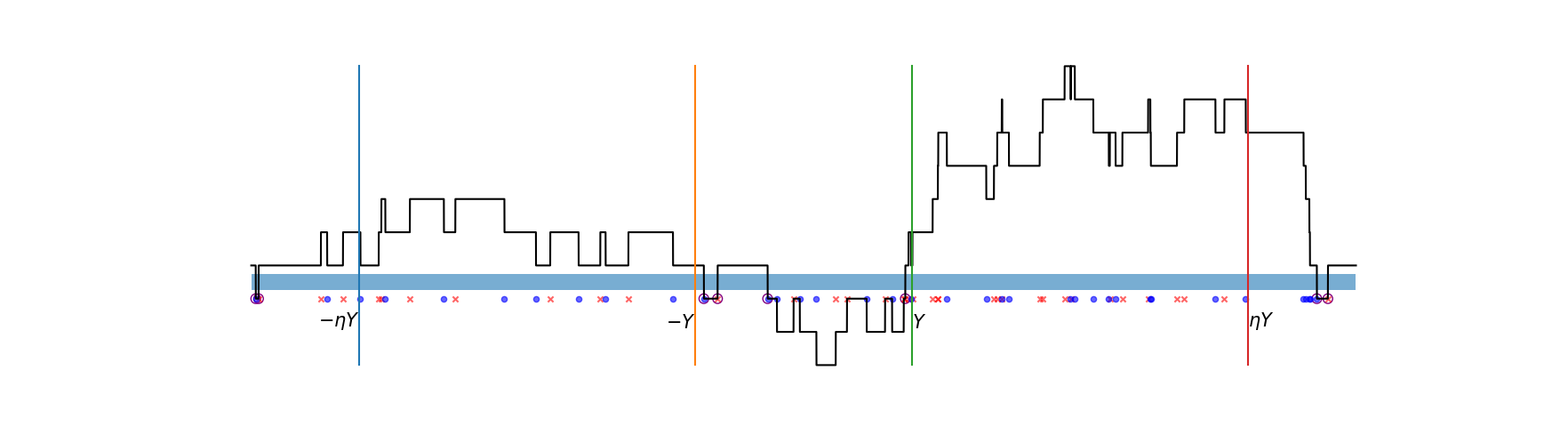}
  \caption{Points that are in the level $\Lambda_0$ and restricted to
  the interval $[-Y,Y]$ must be matched to each other since they are
  all within distance $2Y$ of each other and any other possible point
  they could be matched to is a distance $(\eta-1)Y = 2\kappa Y$ away.
  Matching any two points in $\Lambda_0\cap[-Y,Y]$ outside the
  interval $[-\eta Y,\eta Y]$ contradicts quasistability of
  subcritical $(\gamma,a)$-minimal matchings
  (\cref{quasi}).}\label{fig:rw_eta}
\end{figure}

The proof of \cref{prop:eta} can be found in \cite[Section 7, Lemma
32, Lemma 33]{minimal}, it uses properties of continuous time random
walks to show that for any $\eta$, there will be a value of $Y$ such
that the configuration in \cref{fig:rw_eta} holds.
We will use \cref{prop:eta} with the value of $\kappa$ found in
\cref{quasi} in the rest of this proof.

We are finally ready to prove \cref{compwbal}.

\begin{proof}[Proof of \cref{compwbal}]
  (i) We saw in \cref{super} for $\gamma\in\{1^+\}\cup(1,\infty)$ and
  in \cref{lem:1minus} for $\gamma=1^-$ that regardless of the value
  of $a\in(0,\infty)$, the family of minimal matchings will be given
  by $\{\widetilde{M}_k\}_{k\in\Z}$ and
  $\{M_\infty,M_{-\infty}\}\cup\{M_k\}_{k\in\Z}$ respectively.
  Therefore, the balanced and unbalanced minimal matchings are
  identical for these values of $\gamma$.

  For $\gamma\in\{-\infty,\infty\}$ and fixed finite $a\in(0,\infty)$,
  using the argument in \cref{sec:modelmainresults} we know that
  the minimal matchings will be given by the stable and the altruistic
  matchings respectively, both of which only depend on $\gamma $ and
  not on the value of $a$. \\
  (ii) Fix $\gamma\in(-\infty,1)$ and $a\in(0,\infty)$. From
  \cref{quasi}, we know there exists $\kappa_1 = \kappa(\gamma,a)$ and
  $\kappa_2 = \kappa(\gamma,1)$ such that
  \eqref{kappa-ineq} holds for pairs of points in the unbalanced $(\gamma,a)$-minimal
  matching and the balanced $(\gamma,1)$-minimal matchings
  respectively. If we take $\kappa =
  \max\{\kappa_1,\kappa_2\} $, then with this value of $\kappa$,
  \eqref{kappa-ineq} holds for points in both these matchings. Let
  $\eta=2\kappa+1$.
  Now, we use \cref{prop:eta} to find the corresponding $Y$ such that
  the random walk $W>0$ on $[-\eta Y,-Y]\cup[Y,\eta Y]$.

  Consider the
  level $\Lambda_0$. There are at least two points
contained in the set $\Lambda_0\cap[-Y,Y]$ because $W(Y),W(-Y)>0$
  from our assumption in \cref{prop:eta},
and $W(0^-)=-1$ from
our assumption about the Palm process.
We claim that in both the $(\gamma,a)$-minimal matching and the
    $(\gamma,1)$-minimal matching, all the points in
    $\Lambda_0\cap[-Y,Y]$ are matched to other points in the same set.
    This happens because (1) there is an equal number of red and
blue points in the set $\Lambda_0\cap[-Y,Y]$ and (2) all of these points are
at most a distance of $2Y$ apart from each other, and any other points
they could be matched to are outside the interval $[-\eta Y,\eta Y]$
and hence at least a distance $(\eta-1)Y = \kappa (2Y)$ away. Since
\eqref{kappa-ineq} holds for this value of $\kappa$ for both the
unbalanced $(\gamma,a)$-minimal matching and the balanced
$(\gamma,1)$-minimal matching, no pair of oppositely colored points which are a
distance of $2Y$ from each other will both be matched to
partners further than $2\kappa Y$ away in either the
  $(\gamma,a)$-minimal matching or the $(\gamma,1)$-minimal matching.

This proves that the set $\Lambda_0\cap[-Y,Y]$ will be matched to
itself in both the $(\gamma,a)$-minimal matching and the
$(\gamma,1)$-minimal matching. Therefore, any difference that these
two matchings might have, must also be restricted to this finite set.
\end{proof}

The above proof \cref{compwbal}(ii) has some of the same ideas as the
proof of  \cite[Theorem 3(ii)]{minimal}, which demonstrates that for
$\gamma,\gamma'<1$, the balanced $\gamma$-minimal matching and the
balanced $\gamma'$-minimal matching have finite difference.

\begin{remark}\label{rem:canonical}
  There are several results proven in \cite[Theorem 1,2,3]{minimal}, which hold for
unbalanced costs as well. These results are regarding perfectness and
existence of of the
minimal matchings, classifications of minimal matchings in terms of
their properties and comparisons between minimal matchings for two
different values of $\gamma$. These hold for unbalanced costs as well
with some modifications. The only significant difference is the proof
of \cref{quasi} which proves quasistability.

\end{remark}

\begin{remark} \label{rem:trans}
Observe that finite difference is not a transitive property any pair of the
$(1^-,a)$-minimal matchings
$\{M_\infty,M_{-\infty}\}\cup\{M_k\}_{k\in\Z}$ do not have finite
difference with each other but they have finite difference with any of
the unique
$(\gamma,a)$-minimal matchings for fixed $\gamma<1$ and $a\in(0,\infty)$.
\end{remark}

\section{Limiting values of $a$}\label{sec:lims}

In this section, we discuss $(\gamma,a)$-minimal matchings when we take limits
of $a$ to be either $0$ or $\infty$ and prove
\cref{thm:a-infty,thm:a-zero,thm:weirdlim}.
 In \cref{rem:lims}, we discuss why \cref{thm:a-infty,thm:a-zero} are
 somewhat counterintuitive.

First, we prove \cref{thm:a-infty}. In this theorem, we assume that
$a=\infty$ for the cost functions listed in \eqref{eq:costonedim}.
This implies that for $\gamma\geq 0$, the cost of any given blue-red
edge is infinite and for $\gamma<0$, the cost of any given blue-red
edge is $-\infty$. This is because the cost function is negative for
$\gamma<0$.
In order to find the minimal matching in situations where some edges
may have infinite cost, we restrict the minimization over all
matchings consisting entirely of edges that cost less than infinity.
This might not always be possible for any given point configurations,
but it is possible if the red and blue points are given by Poisson
point processes on $\R$.

\begin{proof}[Proof of \cref{thm:a-infty}]
  (i) For $\gamma\in(-\infty,0)$ and $a=\infty$, the cost function
  is given by
  \[ f_{\gamma,\infty} (x) =
    \begin{cases}
      -\infty, & x\leq 0, \\
      -x^\gamma, & x>0,
    \end{cases}
  \]
  implying that any blue-red edge costs $-\infty$. There are
  uncountably many matchings consisting only of blue-red edges which
  can be constructed by taking the matching $M_\infty$ (constructed in
  \cref{lem:1minus}) and flipping any number of pairs of the nested
  edges into entwined edges so that no edge changes orientation. All
  these matchings will be minimal since they are made of edges which
  cost $-\infty$.
(ii)   We know that for $\gamma=1^-$, there exists a minimal matching
  with only red-blue edges which can be constructed by taking all the
  edges which are below the random walk, i.e., the matching
  $M_{-\infty}$ in \cref{fig:1minus}.   This matching is locally
  infinite and a factor.
  This matching has to be the unique $(1^-,\infty)$-minimal matching since
  any other matching either contains a blue-red edge which has
  infinite weight or entwined edges, which aren't preferred for a
  subcritical matching. Similarly for $\gamma\in[0,\infty)$, we also prefer
  nested edges to entwined edges and since the matching $M_{-\infty}$ contains no
  edges with infinite cost and cannot be improved by finite
  reordering, $M_{-\infty}$ will be the unique $(\gamma,\infty)$-minimal
  matching for $\gamma\in(-\infty,1)\cup\{1^-\}$.

(iii) For $\gamma=1$,
there are uncountably many $(1,\infty)$-minimal
 matchings.  We can construct these
 by taking the matching $M_{-\infty}$ and considering any reordering
 which does not introduce a blue-red edge. As long as all the edges
 are red-blue edges, regardless of if they are entwined or nested, the
 cost of the matching remains unchanged. This gives us an uncountable
 family of matchings with the same cost for any comparable
 collection of edges. These are all minimal since any other matching
 not included has an edge of infinite cost.

  (iv) For $\gamma>1$, first we claim that the $(\gamma,\infty)$-minimal
  matching cannot have an unmatched point. And then argue that there
  cannot be a matching with every point matched.

  \begin{leftbar}
    \begin{claim}
     For $\gamma>1$, a $(\gamma,\infty)$-minimal
  matching cannot have an unmatched point.
\end{claim}
\begin{proof}
  Say any given $(\gamma,\infty)$-minimal matching has an unmatched point.
  WLOG we assume there is an unmatched red point at $0$. Let $W$
  denote the
  continuous time random walk representation of the two Poisson point
  processes. There exists $x>0$ such that $W(x)<0$. This implies in
  the interval $(0,x)$, the number of blue point exceeds the number of
  red points. Say there is a blue point at $y$ such that it does not
  have a partner in the interval $(0,x)$. This blue  point at $y$ must be matched to the left of $0$, since it
  otherwise the cost of the edge would be infinite. Hence, our minimal
  matching can be improved by matching the blue point at $y$ to the
  red point at $0$.

  A similar argument for an unmatched blue point shows we cannot have
  any unmatched points.
\end{proof}
  \end{leftbar}

  We know that for $\gamma>1$, the minimal matchings must preserve
  relative orderings of the red and blue points. There is no way of
  doing this in a way which only has red-blue edges since this would
  correspond to the continuous
  time random walk having a lower bound.
\end{proof}

\begin{remark} \label{rem:lims}
  Notice that the behaviour of the $(\gamma,a)$-minimal matchings as
  $a\rightarrow 0$ and $a\rightarrow\infty$ is very different from
  what we observe for
  the $(\gamma,0)$ and $(\gamma,\infty)$-minimal matchings.
  \begin{enumerate}[(i)]
\item   This is clearly seen for supercritical $\gamma$, i.e.,
$\gamma\in\{1^+\}\cup(1,\infty)$, where for all finite $a$, there is a
fixed countable family of $(\gamma,a)$-minimal matchings however there does not
exist even one $(\gamma,\infty)$-minimal matching. Furthermore, the
family of subcritical $(\gamma,a)$-minimal matchings is a fixed countable
family of non-factor matchings whereas, the $(\gamma,\infty)$-minimal
matching is a unique factor matching.
\item Another possibly counterintuitive observation is that  setting
  $a=0$ and $a=\infty$ result in very different minimal matchings.
  One might initially suspect that these might have
similar properties since the two cost functions $f_{\gamma,a}$ and
$f_{\gamma,1/a}$ are simply reflections and rescalings of each other and
because \cref{fig:scatter} with the $a$-axis plotted on the $\log$ scale
appears to be symmetric about $a=1$.
However, this is not the case.
\end{enumerate}
This is because in
\cref{thm:a-infty} and \cref{thm:a-zero}, we are minimizing over
different sets. In \cref{thm:a-infty}, for $\gamma>0$ we consider matchings
consisting only of red-blue edges whereas in \cref{thm:a-zero}, we
consider all possible matchings.
\end{remark}

Below, we prove \cref{thm:a-zero}, which considers a cost function
which assigns cost zero to all blue-red edges.

\begin{proof}[Proof of \cref{thm:a-zero}]
  (i) For $\gamma\in(0,\infty)\cup\{1+\}\cup\{1-\}$, the cost function
  is given by
  \[
 f_{\gamma,0}(x) =
    \begin{cases}
      0, & x<0,\\
      x^\gamma, & x\geq 0. \end{cases}
      \]
We consider the matching $M_{\infty}$, previously described in the proof of
      \cref{thm:a-infty}(i) and \cref{fig:1minus}. This locally
      finite, factor matching is constructed by taking the
      edges which lie above the random walk \cref{walk}. Since this
      matching consists only of blue-red edge, it has
      cost zero, which is the lowest possible cost, implying that it must be a
      $(\gamma,0)$-minimal matching for $\gamma>0$.

      Furthermore, we may swap any
      pair of nested edges with the same orientation to get entwined
      edges with the same zero cost. Any reordering which does not
      introduce a red-blue edge works. Since there are infinitely many
      such edges which can be swapped, we get uncountably many
      $(\gamma,0)$-minimal
      matchings for $\gamma\in(0,\infty)\cup\{1^-,1^+\}$.
      (ii) For $\gamma=0$, the cost function is given by
       \[
 f_{\gamma,0}(x) =
    \begin{cases}
      0, & x<0,\\
      \log x, & x\geq 0.\\ \end{cases}
      \]
Notice that the range of the two branches of the function ($x>0$ and
      $x<0$) have different ranges but a non-zero intersection. We
      show below that such a cost function cannot be
      scale invariant. Consider a configuration of points with two
      blue points and a red point in the middle. For smaller scales,
      the cost of the red-blue edge will be negative, and for larger
      scales, the cost of this edge will be positive, whereas the
      blue-red edge will always have zero cost. Since rescaling the
      points changes the minimal matching, the cost function is not
      scale invariant.

      Having said this, we attempt to find a family of minimal matchings below.
      For this cost function, blue-red edges shorter than length $1$
      are preferred over red-blue edges which are preferred over
      blue-red edges of length greater than $1$. Among the red-blue
      edges with length smaller than 1, all of whom have negative
      cost, we prefer shorter edges since these have a smaller cost.
      However all blue-red edges have cost equal to $0$, hence we
      prefer these to any given red-blue edge of length greater than
      $1$, which have positive cost.

      A matching can be constructed in the following way. From each
      red point $r$, we consider a growing interval given by $I^r_t =
      (r,r+t)$ at time $t<1$. We may match a red point to the first
      blue point which enters it's interval before time $t$. For the
      points unmatched at time $t=1$, we can construct a matching
      consisting entirely of blue-red edges according to the algorithm
      from \cref{lem:1minus} used to construct $M_{\infty}$.
      Similarly, this could be done for a different threshold, for
      example where the cost function looks like
      \[ f(x) =
        \begin{cases}
          0,&x\leq 0,\\
          \log(bx),&x>0.
        \end{cases}
      \]
      The minimal matching here would not be scale invariant, but it can
      be constructed using the same method as above.

      (iii) For $\gamma\in(-\infty,0)$, the cost function looks
      like \[  f_{-\gamma,0}(x)=
           \begin{cases}
             0 ,& x<0,\nonumber\\             - x^{-\gamma}, & x\geq
             0.
           \end{cases} \]
       In order to minimize this cost
                                 function, we would like to include as
                                 many red-blue edges as possible since
                                 these have negative costs. And among
                                 those, we want the smallest edges
                                 since these bring the cost down even
                                 more. Furthermore, recall that for
       subcritical $\gamma$, nested edges are always preferred over entwined
       edges.

Therefore, the minimal matching is given by $M_{-\infty}$, seen
in \cref{fig:1minus}. This matching consists only of red-blue edges
and also prioritizes minimizing the shortest edges.
\end{proof}

Finally for the proof of theorem \cref{thm:weirdlim}, we consider a
configuration of four points as in \cref{fig:weirdlim} with four
points alternating in color with consecutive distances $x$, $y$ and
$z$. This characterizes the matchings we get along different rays in
\cref{fig:scatter}.

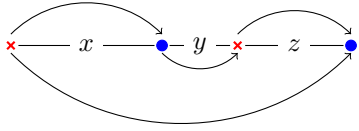
\begin{figure}[h]
  \centering
 % \usetikzlibrary{shapes.geometric}
 % \usetikzlibrary{plotmarks} % LaTeX and plain TeX when using TikZ
  \begin{tikzpicture}
    \draw[mark=x, color=red,thick] plot coordinates {(0,0)};
        \draw[mark=x, color=red,thick] plot coordinates {(3,0)};
%    \draw[red,thick,mark=asterisk] at (-1,0);
%     \node[draw,fill, color=red, star,star points=7,inner sep = 0.3ex] at (0,0) {};
    % \filldraw [red] (0,0) circle (2pt);
 %    \node[draw,fill, color=red, star,star points=7,inner sep = 0.3ex] at (3,0) {};
        \filldraw [blue] (2,0) circle (2pt);
    %\filldraw [red] (3,0) circle (2pt);
    \filldraw [blue] (4.5,0) circle (2pt);
    \draw[-] (0+0.1,0) -- (2-0.1,0) node[midway,fill=white]{$x$} ;
    \draw[-] (2+0.1,0) -- (3-0.1,0) node[midway,fill=white]{$y$} ;
    \draw[-] (3+0.1,0) -- (4.5-0.1,0) node[midway,fill=white]{$z$} ;
    \draw[->] (0,0.15) [out=45,in=135] to (2,0.15);
    \draw[->] (3,0.15) [out=45,in=135] to (4.5,0.15);
    \draw[->] (0,-0.1) [out = 315,in=225] to (4.5,-0.1);
    \draw[->] (2,-0.1) [out=315,in=225] to (3,-0.1);
    %\node at (5.9,0.75) [align=left] {entwined edges/matching};
    %\node at (5.9,-0.75) [align=left] {nested edges/matching};
  \end{tikzpicture}
  \caption{For $\gamma\rightarrow -\infty$ and $a=\alpha^\gamma$, the
  minimal matching will be the one which includes the edge with the
  smallest value among $|x|, |z|,\alpha|y|$ and $|x+y+z|$.
  And for $\gamma\rightarrow \infty$ and $a=\alpha^\gamma$, the
  minimal matching avoids the largest of these four values.
  } \label{fig:weirdlim}
\end{figure}

\begin{proof}[Proof of \cref{thm:weirdlim}]
  (i) Since we are taking the limit as $\gamma\rightarrow -\infty$, we
  are in the subcritical regime which never contains entwined edges
  and only contains nested or separate edges (see \cref{possibilities}).

  Say $a=\alpha^\gamma$, then the cost of
  the nested edges in \cref{fig:weirdlim} will be
  $-(x+y+z)^\gamma-(\alpha y)^\gamma$ and
  the cost of the separate edges will be $-
  x^\gamma- z^\gamma$. Since we are taking $\gamma\rightarrow -\infty$, the
  term which is smallest out of $|(x+y+z)|$, $|\alpha y|$, $| x|$ and
  $|z|$ will have the largest negative contribution to the
  cost, so it must be included in the matching. The remaining terms
  will be irrelevant.
  If we let $c(x,y) $ denote the ``cost'' given by
  \begin{equation*}
  c(x,y) = [\mathbbm{1}_{x\leq
  y}+\alpha\mathbbm{1}_{x>y}]|x-y| = \begin{cases} \alpha|x-y|, &
  y<x, \\ |x-y|, & x\leq y,\end{cases}
\end{equation*}
the matching above simply minimizes the smallest value of this cost
function for the possible edges.
If we have two points $x$ and $y$ not matched to each other then it
must be because either $c(x,M(x))$ is smaller than $c(x,y)$ or
$c(y,M(y))$ is smaller than $c(x,y)$. Therefore we recover the
$\alpha$-stable matching. Finally, since the $\alpha$-stable matching
satisfy quasistability (\cref{quasi}), we can use the
argument in the proof of \cite[Proposition 26]{minimal}, to prove these matchings
exist, and are invariant and perfect.

(ii)
We use the same argument to recover the $\alpha$-altruistic matching. In this case, the largest term
out of  $|x+y+z|$, $|\alpha y|$, $| x|$ and $| z|$ will
dominate the sum, hence we aim to minimize the largest of these values
resulting in the $\alpha$-altruistic matching.
Note that for $d=1$ the set of $\alpha$-altruistic minimal matchings will
be identical to the countable set of supercritical minimal matchings
(pictured in \cref{fig:super}). This is due to \cref{super}, which argues that for any
$\gamma>1$, the matching prefers separate edges over entwined edges
over nested edges.
\end{proof}

\begin{remark}\label{rem:alphastab}
The definitions for $\alpha$-stable matchings and $\alpha$-altruistic
  matchings (\cref{def:alphastab,def:alphaalt}) and the result
  \cref{thm:weirdlim} can be extended to higher dimensions as well.
For stable matchings from \cite{minimal} we can consider growing balls centered at every red
  point of radius $t$ at time $t$ and match each red point to the
  first unmatched blue point which enters its ball. This gives us the
  stable matching (see \cite[Lemma 8]{minimal}).
  Similarly, for the $\alpha$-stable matching between points in
  $\R^d$, we can think of a shape given by $r+t A$ at time $t$ where
  $A$ is a fixed set and match the red point at $r$ to the first blue
  point to enter this shape.

\end{remark}

\bibliographystyle{unsrt}
\bibliography{references.bib}

\end{document}